\documentclass[11pt, oneside, reqno]{amsart}
\usepackage[margin=2.5cm]{geometry}
\usepackage{tikz-cd}
\usepackage{amssymb,amsrefs, mathrsfs, bm, bbm, enumitem, tensor}
\usepackage{mathtools}
\usepackage[draft=false]{hyperref}
\usepackage[noabbrev,capitalise,nameinlink]{cleveref}
\usepackage{bookmark}
\hypersetup{                                                         
  pdfauthor={},
  pdftitle={},
  colorlinks=true, 
  linkcolor={},
    citecolor={},
    urlcolor={},
    linktocpage=true                                                      
}

\newcommand{\Gu}{G^{(0)}}
\newcommand{\cg}{\mathscr{C}_c(G)}
\newcommand{\im}{{\rm{im}}}

\newcommand{\ess}{{\rm{ess}}}
\newcommand{\tiG}{\tilde{G}}
\newcommand{\tiGu}{\tilde{G}^{(0)}}
\newcommand{\mfi}{\mathfrak i}
\newcommand{\abs}[1]{\left\vert#1\right\vert}
\newcommand{\norm}[1]{\left\|#1\right\|}
\newcommand{\supp}{{\rm supp}}
\newcommand{\osupp}{\supp^{\circ}}

\newcommand{\PRLsep}{\noindent\makebox[\linewidth]{\resizebox{0.3333\linewidth}{1pt}{$\bullet$}}\bigskip}

\newtheorem{thm}{Theorem}[section]

\newtheorem{cor}[thm]{Corollary}

\newtheorem{lem}[thm]{Lemma}
\newtheorem{prp}[thm]{Proposition}
\newtheorem{question}[thm]{Question}

\newtheorem{introthm}{Theorem}

\newtheorem{introcor}[introthm]{Corollary}

\newtheorem{introquestion}{Question}

\theoremstyle{definition}
\newtheorem{dfn}[thm]{Definition}

\theoremstyle{remark}
\newtheorem{rmk}[thm]{Remark}

\numberwithin{equation}{section}

\makeatletter
\@namedef{subjclassname@2020}{%
  \textup{2020} Mathematics Subject Classification}
\makeatother

\begin{document}
\date{\today}
\title[On dense subalgebras of the singular ideal in groupoid C*-algebras]{
On dense subalgebras of the singular ideal in groupoid C*-algebras}

\author[Gonzales]{Julian Gonzales}
\address{Julian Gonzales, School of Mathematics and Statistics  \\
University of Glasgow\\ University Place \\ Glasgow \\ G12 8QQ, United Kingdom}
\email{Julian.Gonzales@glasgow.ac.uk}

\author[Hume]{Jeremy B. Hume}
\address{Jeremy B. Hume, School of Mathematics and Statistics  \\
Carleton University\\ 4302 Herzberg Laboratories\\ Ottawa, ON\\ K1S 5B6, Canada}
\email{jeremybhume@gmail.com}

\begin{abstract}
We prove that ideals in amenable second-countable non-Hausdorff \'etale groupoid $C^*$-algebras are determined by their isotropy fibres. As an application, we characterise when the singular functions in Connes' algebra are dense in the singular ideal in terms of a property of explicit ideals in the isotropy group $C^*$-algebras. We then show this density property holds for all $C^*$-algebras of groupoids with finite-by-nilpotent isotropy groups.
\end{abstract}



\maketitle

\vspace{-0.4cm}
\section*{Introduction}

In recent years non-Hausdorff groupoids and their $C^*$-algebras have gained increased attention. Many groupoids arising from dynamics and geometry are non-Hausdorff, and it is therefore important to develop a robust theory. Indeed, there are natural examples of groupoids of germs, groupoids arising from self-similar groups, and groupoids arising from foliations that are non-Hausdorff (\cite{Con82}, \cite{MM03}, \cite{Nek05}, \cite{Nek09}).

In contrast with the Hausdorff case, the functions in the reduced $C^*$-algebra $C^*_r(G)$ of a non-Hausdorff groupoid $G$ are not necessarily continuous. In fact, there can exist (non-zero) functions in $C^*_r(G)$ whose set of non-zero values has empty interior in $G$. The set of all such functions is known as the \emph{singular ideal} $J$ in $C^*_r(G)$. Historically, the singular ideal has been an obstruction to understanding simplicity for reduced $C^*$-algebras of non-Hausdorff groupoids (\cite{CEPSS}). Characterisations of simplicity and the ideal intersection property have been obtained for the quotient by this ideal (\cite{CEPSS}, \cite{KKLRU}), which is known as the \emph{essential groupoid $C^*$-algebra} (\cite{KM21}, \cite{EP22}). Therefore, the singular ideal is the only obstacle to understanding these properties for $C^*_r(G)$.

Much work has been done to find conditions that ensure that the singular ideal $J$ vanishes (see \cite{CEPSS}, \cite{KM21}, \cite{NS23}, \cite{GNSV}). In \cite[Theorem~A]{BGHL} (see also \cite[Theorem~F]{Hume}) it is characterised when the singular ideal has trivial intersection with the underlying groupoid $^*$-algebra $\cg$, known as \emph{Connes' algebra} (first defined in \cite{Con82} for groupoids arising from foliations). Recently, \cite[Theorem~B]{Hume} characterised vanishing of the singular ideal in terms of vanishing of explicit ideals in the isotropy group $C^*$-algebras.

When the singular ideal is non-zero, little is known about its structure. For example, it is not known in general whether $J\neq \{0\}$ implies $J\cap\cg \neq \{0\}$. This implication was shown to hold for groupoids that are ``finitely non-Hausdorff'' in \cite[Theorem~C]{BGHL}. In \cite{Hume}, the problem was reduced to one about group $C^*$-algebras, and this led to a positive answer in many cases (\cite[Theorem~H]{Hume}).

In order to understand the structure of $J$, it is necessary to ask the following deeper question (see \cite[Question~4.11(III)]{BGHL}): for which groupoids is $J \cap \cg$ dense in $J$? In recent work (\cite{MS25}), it was shown that there exist non-amenable \'etale groupoids for which $J \cap \cg$ is not dense in $J$. The following remains open however.
\begin{introquestion}\label{question:dense?}
    Is $J \cap \cg$ dense in $J$ for any amenable \'etale groupoid $G$?
\end{introquestion}
A positive answer to this question would allow progress to be made towards understanding the structure of $J$ (since the functions in $\cg$ are easier to understand). It is therefore natural to ask a follow up question: can we explicitly describe elements of $J \cap \cg$? 

\vspace{0.3cm}
The goal of our paper is to study these questions following the approach in \cite{Hume} of studying the singular ideal through its isotropy fibres. Let us outline the main achievements and methods of the paper. The first four points listed here are for amenable and second-countable \'etale groupoids.
\vspace{0.3cm}
\begin{itemize}
    \item We show that ideals in a (non-Hausdorff) groupoid $C^*$-algebra are determined by their isotropy fibres. We derive this result from the Hausdorff case (proven in \cite{CN24}) through a novel application of the Hausdorff cover (see \cite{BGHL}).
    \vspace{0.2cm}
    \item We use this result to characterise when $J \cap \cg$ is dense in $J$ in terms of a property of explicit ideals in the isotropy group $C^*$-algebras. This reduces Question~\ref{question:dense?} to a property of discrete amenable groups, which we call the \emph{Density Property}. 
    \vspace{0.2cm}
    \item We show the Density Property is preserved under various operations on the group. We pair this with structural properties of the space of subgroups (equipped with the Chabauty topology) to 
    prove that the Density Property holds for all \textbf{finite-by-nilpotent groups}.
    \vspace{0.2cm}
    \item This reduction allows us to prove that $J \cap \cg$ is dense in $J$ for groupoids with finite-by-nilpotent isotropy groups, groupoids arising from contracting self-similar group actions, and for amenable \'etale group bundles.
    \vspace{0.2cm}
    \item We describe an explicit family of functions whose span is always dense in $J \cap \cg$.
    \vspace{0.2cm}
    \item For \emph{ample} groupoids, we completely describe the structure of functions in the \emph{algebraic singular ideal} $J_\mathbb{C}$, and prove $J_{\mathbb{C}}$ is dense in $J \cap \cg$.
\end{itemize}

\PRLsep

Let us state our main results in more detail. The first four results describe the progress made towards answering Question~\ref{question:dense?}. In \cite{CN24} it was shown that ideals $I \trianglelefteq C^*(G)$ in a groupoid $C^*$-algebra are mapped to ideals in the isotropy group $C^*$-algebras by the restriction maps $C^*(G) \to C^*(G^x_x)$. These images are known as the \emph{isotropy fibres} of the ideal, and are denoted $(I_x)_{x \in \Gu}$. For amenable second-countable Hausdorff \'etale groupoids, ideals in the groupoid $C^*$-algebra are determined by their isotropy fibres (see \cite[Theorem~2.10]{CN24}). We generalise this result to non-Hausdorff groupoids.

\begin{introthm}[See Theorem \ref{IdealFib}]\label{IntroIdealFib}
    Let $G$ be an amenable and second-countable (non-Hausdorff) \'etale groupoid and suppose $I$ and $K$ are ideals in $C^*(G)$. Then, $I = K$ if and only if $I_{x} = K_{x}$ for all $x\in \Gu$.
\end{introthm}

We apply this result to the singular ideal $J$ in the reduced $C^*$-algebra of a non-Hausdorff groupoid (see Proposition \ref{DensChar}).
The isotropy fibres $J_{x}$ are calculated explicitly in \cite[Theorem~5.5]{Hume}, where it is shown that $J_{x}$ is equal to an intersection of kernels of quasi-regular representations associated with subgroups of $G^{x}_{x}$ in a certain collection $\mathcal{X}(x)$. This ideal is denoted $J_{G^x_x, \mathcal{X}(x)} \trianglelefteq C^*(G^x_x)$. We refer the reader to Definition~\ref{dfn:mathcalX(x)} and Section \ref{sub:CharDens} for the definitions of $\mathcal{X}(x)$ and $J_{G^x_x, \mathcal{X}(x)}$. 

\begin{introthm}[See Theorem \ref{ReductionToIso}]\label{IntroReduct}
    Let $G$ be an amenable and second-countable \'etale groupoid. Then $J\cap\cg$ is dense in $J$ if and only if $J_{G^x_x, \mathcal{X}(x)} \cap \mathbb{C}[G^x_x]$ is dense in $J_{G^x_x, \mathcal{X}(x)}$ for all $x \in \Gu$. In particular, if all the isotropy groups $G^x_x$ satisfy the Density Property, then $J\cap\cg$ is dense in $J$.
\end{introthm}

In Theorem~\ref{IntroReduct}, the \emph{Density Property} is a property of discrete amenable groups introduced in Definition \ref{DefnPropD}. Finite groups trivially satisfy the Density Property, and in Theorem \ref{thm:bigthm} we show that all finite-by-nilpotent groups have the Density Property. Currently, we do not know of any discrete amenable group that fails to have the Density Property. We remark that it is a consequence of Theorem~\ref{IntroReduct} and a non-Hausdorff groupoid construction in \cite[Section~6]{Hume} that if there exists such a group, then there also exists an amenable and second-countable \'etale groupoid for which $J\cap\cg$ is not dense in $J$ (see Corollary \ref{equivalence}).

By the definition of the ideals $J_{G^x_x, \mathcal{X}(x)}$, Theorem~\ref{IntroReduct} asserts that density of $J \cap \cg$ in $J$ can be understood entirely in terms of quasi-regular representations of isotropy groups. In light of this, we are able to describe some explicit classes of groupoids for which $J \cap \cg$ is dense in $J$.

\begin{introthm}[See Corollary \ref{cor:bigcor}]\label{Intro:bigthm}
    Let $G$ be an amenable and second-countable \'etale groupoid. For each $x\in \Gu$ let $\mathcal{X}(x)$ be as in Definition~\ref{dfn:mathcalX(x)}, and assume that one of the following holds.
    \begin{enumerate}[label=(\Roman*)]
        \item The isotropy group $G_{x}^{x}$ is finite-by-nilpotent.
        \item The subgroups $X\in \mathcal{X}(x)$ are all finite.
        \item The subgroups $X\in \mathcal{X}(x)$ are all co-finite in $G_{x}^{x}$.
        \item The subgroups $X \in \mathcal{X}(x)$ are all normal in $G^x_x$.
    \end{enumerate}
    Then $J \cap \cg$ is dense in $J$.
\end{introthm}

Here, a countable group $\Gamma$ is \emph{finite-by-nilpotent} if there exists a finite normal subgroup $N \trianglelefteq \Gamma$ for which the quotient $\Gamma/N$ is nilpotent. Theorem~\ref{Intro:bigthm} covers many important classes of groupoids. In particular, any groupoid of germs associated with the action of a countable finite-by-nilpotent group will satisfy (I) for all $x \in \Gu$. Condition (II) is satisfied for all $x \in \Gu$ whenever the closure of the unit space of $G$ has finite source and range fibres. This is the hypothesis assumed in \cite[Theorem~C]{BGHL} when proving the weaker property that $J\neq \{0\}$ implies $J\cap\mathscr{C}_{c}(G)\neq \{0\}.$

\vspace{0.3cm}

The next result considers \'etale group bundles -- that is, \'etale groupoids whose range and source maps are equal. Historically, \'etale group bundles have proven to be a good resource for providing counterexamples to conjectures (see \cite{HLS}, \cite{Wil15} for example). It was shown in \cite{MS25} that there exist non-amenable \'etale group bundles for which $J \cap \cg$ is not dense in $J$. A direct corollary of Theorem~\ref{Intro:bigthm} (IV) is that density always holds for \emph{amenable} \'etale group bundles.

\begin{introcor}[See Corollary~\ref{cor:BundleOfGroups}]
    If $G$ is an amenable and second-countable \'etale group bundle, then $J \cap \cg$ is dense in $J$.
\end{introcor}

\PRLsep

\vspace{-0.1cm}
The remaining results  describe the structure of elements in $J \cap \cg$. In Section~\ref{sec:StructureJcg} we describe an explicit set whose span is dense in the subalgebra $J \cap \cg$. Elements in this spanning set are weighted sums of functions $\psi \in C_c\big(\Gu\big)$ lifted by the source map, and are built using purely topological and algebraic data coming from the groupoid and its isotropy groups. This provides an explicit dense subspace of the singular ideal $J$ for groupoids as described in Theorem~\ref{Intro:bigthm}. 

\begin{introthm}[See Theorem~\ref{SpanJcg}]\label{Intro:SpanJcg}
Let $G$ be an \'etale groupoid that is covered by countably many open bisections. For $x\in \Gu$ let $b \colon G^{x}_{x}\to\mathbb{C}$ be a function with finite support $g_1, \dots, g_n$ such that

\begin{equation}
\label{eqn:ideal}
\sum_{g\in kX} b(g) = 0 \text{ for all } k\in G^{x}_{x} \text{ and } X\in\mathcal{X}(x).
\end{equation}
Then, given any $h \in G_x$ and any open bisections $V_1, \dots, V_n$ containing $hg_1, \dots, hg_n$, there exists an (explicitly defined) open neighbourhood $W\subseteq \bigcap_{i=1}^n s(V_{i})$ of $x$ such that 
\begin{equation}
\label{eqn:construction}
f^{\psi} \coloneq \sum_{i=1}^n b(g_i)(\psi\circ s|_{V_{i}})
\end{equation}
belongs to $J \cap \cg$ for all $\psi\in C_{c}(W)$. Furthermore, the linear span of such functions $f^{\psi}$ is dense in $J\cap\cg$ with respect to any $C^*$-norm.
\end{introthm}

We remark that an element $b \in \mathbb{C}[G^x_x]$ satisfies \eqref{eqn:ideal} if and only if, for every subgroup $X \in \mathcal{X}(x)$, $b$ belongs to the kernel of the associated quasi-regular representation $\lambda_{G^x_x/X}$.

\PRLsep

We now turn our attention to non-Hausdorff \emph{ample} groupoids, where one can intersect $J$ with the complex Steinberg algebra $\mathbb{C}G$ (see \cite{S10}) to obtain what is known as the \emph{algebraic singular ideal} $J_\mathbb{C}$. A characterisation for the vanishing of $J_\mathbb{C}$ is given in \cite{BGHL}, and together with \cite{CEPSS} and \cite{SS21} this leads to a complete characterisation of simplicity for Steinberg algebras (see also \cite{SS23} and \cite{GNSV}). In this paper we look beyond the vanishing of $J_\mathbb{C}$ and describe an explicit spanning set analogous to that in Theorem~\ref{Intro:SpanJcg}. We stress the following: the structure of functions in the algebraic singular ideal $J_\mathbb{C}$ is now fully understood.

\begin{introthm}[See Theorem~\ref{SpanAlgSing}]\label{Intro:SpanAlgSing}
Let $G$ be an ample groupoid that is covered by countably many open bisections. Let $x, b, h, V_1, \dots, V_n$ be as in Theorem~\ref{Intro:SpanJcg}. There exists a compact open neighbourhood $W\subseteq \bigcap_{i=1}^n s(V_{i})$ of $x$ such that the function 
\begin{equation}
\label{eqn:AlgSingconstruction}
f^{W} \coloneq \sum_{i=1}^n b(g_i) \mathbbm{1}_{U_i}
\end{equation}
belongs to $J_\mathbb{C}$, where $U_i \coloneq V_i \cap s^{-1}(W)$. Furthermore, $J_{\mathbb{C}}$ is equal to the linear span of functions of the form $f^{W}$.
\end{introthm}

For ample groupoids, one can ask the following density question: for which ample groupoids is $J_\mathbb{C}$ dense in $J$? Indeed, this is presented as an open problem in \cite{GNSV}. The following theorem implies that $J_\mathbb{C}$ is dense in $J$ if and only if $J \cap \cg$ is dense in $J$.

\begin{introthm}[See Theorem \ref{AlgSingId}]
     Let $G$ be an ample groupoid. Then $J_{\mathbb{C}}$ is dense in $J \cap \cg$ with respect to any $C^*$-norm.
\end{introthm}

An important class of ample groupoids are those arising from self-similar groups (see \cite{Nek05}). These groupoids can be non-Hausdorff, and can also have non-vanishing singular ideal. For a groupoid arising from a \emph{contracting} self-similar group, it was shown in \cite{GNSV} that $J_\mathbb{C} = \{0\}$ if and only if $J = \{0\}$. Using our results, we show that $J_\mathbb{C}$ is in fact dense in $J$ for all groupoids in this class.

\begin{introcor}[See Corollary \ref{SelfSim}]\label{IntroSelfSim}
    Let $G$ be the groupoid arising from a contracting self-similar group action. Then $J_{\mathbb{C}}$ is dense in $J$.
\end{introcor}

In \cite{MS25}, examples of groupoids arising from self-similar actions on infinite alphabets are described for which $J_\mathbb{C}$ is not dense in $J$.

\PRLsep

The paper is structured as follows. In Section~\ref{Preliminaries}, we give the necessary background on the Hausdorff cover groupoid and on isotropy fibres of ideals. In Section~\ref{sec:Ideals} we prove Theorem~\ref{IntroIdealFib}, and then apply this result to the singular ideal $J$ in Section~\ref{sub:CharDens}. Section~\ref{sec:GrpPropD} studies the Density Property for groups and ends by describing some classes of groupoids for which $J \cap \cg$ is dense in $J$. Sections~\ref{sec:StructureJcg} studies the structure of functions in $J \cap \cg$. Section~\ref{sec:Ample} is exclusively concerned with ample groupoids and describes the structure of functions in the algebraic singular ideal $J_\mathbb{C}$.

\section*{Acknowledgments}
The authors would like to thank the Isaac Newton Institute for Mathematical Sciences, Cambridge, for support and hospitality during the programme \emph{Topological groupoids and their $C^{*}$-algebras}, where work on this paper was undertaken. This work was supported by EPSRC grant EP/V521929/1. Hume was supported by NSERC Discovery Grant RGPIN-2021-03834. Gonzales was supported by EPSRC grant EP/T517896/1. Gonzales would like to thank Xin Li for helpful discussions on topics related to this paper.

\section{Preliminaries}\label{Preliminaries}

\subsection{\'Etale groupoids}\label{prelim:etale_gpd}
We refer the reader to \cite{Ren} for formal definitions. A \emph{groupoid} is a small category whose morphisms are all invertible. We let $G$ denote the set of morphisms, and identify the set of objects $\Gu$ with the set of identity morphisms, so that $\Gu \subseteq G$. We call $\Gu$ the \emph{unit space}, and call its elements \emph{units}. For a morphism $g \colon u\to v$, we set $s(g) = u$, $r(g) = v$ and call $s \colon G \to \Gu$ and  $r \colon G \to \Gu$ the \emph{source} and \emph{range} maps, respectively. Composition in the category is then encoded by the \emph{product map} $G\tensor[_s]{\times}{_r} G\to G$, $(g,h) \mapsto gh$. We denote the \emph{inverse} of $g\in G$ by $g^{-1}$. For $x \in \Gu$, we will write $G_x \coloneq s^{-1}(\{x\})$ for the \emph{source fibre} at $x$, $G^x \coloneq r^{-1}(\{x\})$ for the \emph{range fibre} at $x$, and $G^x_x \coloneq G_x \cap G^x$ for the \emph{isotropy group} at $x$.

A \emph{topological groupoid} $G$ is a groupoid equipped with a topology such that the product and inverse maps are continuous. By an \emph{\'etale groupoid} $G$, we mean a locally compact topological groupoid for which the unit space $\Gu$ is Hausdorff, and the range $r \colon G \to \Gu$ and source $s \colon G \to \Gu$ maps are local homeomorphisms. Note that this ensures that the unit space $\Gu$ is open in $G$. This paper will only consider \'etale groupoids, and primarily concerns non-Hausdorff groupoids (all results also hold for Hausdorff groupoids, but often for trivial reasons i.e. because the singular ideal vanishes for Hausdorff groupoids).

An open subset $U \subseteq G$ will be called an open \emph{bisection} if the restrictions $r|_U \colon U \to r(U)$ and $s|_U \colon U \to s(U)$ are homeomorphisms. Note that the open bisections form a basis for the topology of any \'etale groupoid. We define \emph{Connes' algebra} as follows:
\[
\cg \coloneq \text{span}\{f \colon G \to \mathbb{C} \colon f|_U \in C_c(U) \text{ and } f|_{G \setminus U} = 0 \text{ for some open bisection } U \subseteq G\}.
\]
If the groupoid $G$ is Hausdorff, then $\cg = C_c(G)$, the set of all compactly supported continuous functions on $G$. If $G$ is not Hausdorff, then the functions in $\cg$ need not be continuous. We equip $\cg$ with the structure of a $^*$-algebra as follows. For $f, f_1, f_2 \in \cg$, and $g \in G$ define
\[
f^*(g) \coloneq \overline{f(g^{-1})} \quad \text{and} \quad f_1 * f_2 (g) \coloneq \sum_{h \in G_{s(g)}} f_1(gh^{-1})f_2(h).
\]
We now describe the reduced and full groupoid $C^*$-algebras. Given $x \in \Gu$, let $\lambda_x \colon \cg\to B(\ell^2(G_x))$ be the representation given by $\lambda_{x}(f)(\xi) = f * \xi$ for $f \in \cg$ and $\xi \in \ell^2(G_x)$, where $(f * \xi)(g) =  \sum_{h \in G_{x}} f(gh^{-1}) \xi(h)$ for $g \in G_x$. Then the \emph{reduced groupoid $C^*$-algebra} $C^*_r(G)$ of $G$ is defined to be the completion of $\cg$ with respect to the norm $\Vert f \Vert_{C^*_r(G)} = \sup_{x \in \Gu} \Vert \lambda_{x}(f) \Vert$. The \emph{full groupoid $C^*$-algebra} $C^*(G)$ of $G$ is defined to be the completion of $\cg$ with respect to the norm $\Vert f \Vert_{C^*(G)} = \sup_{\pi} \Vert \pi(f) \Vert$, where the supremum is taken over all $^*$-algebra representations $\pi \colon \cg \to B(H)$, for some Hilbert space $H$.

Elements of the reduced groupoid $C^*$-alegbra $C^*_r(G)$ can be viewed as bounded Borel functions on $G$ in a canonical way (see \cite[Lemma~3.15]{BM25}, \cite[Proposition~II.4.2]{Ren}). From now on, we will always treat elements of $C^*_r(G)$ as functions on $G$ in this way. Given $a \in C^*_r(G)$, we define
\[
\osupp(a) \coloneq \{g \in G \colon a(g) \neq 0\}.
\]
\begin{dfn}[\cite{CEPSS}, \cite{KM21}, \cite{EP22}]
    The \emph{singular ideal} $J$ in $C^*_r(G)$ is defined as
    \[
    J \coloneq \{a \in C^*_r(G) \colon \osupp(a) \text{ has empty interior}\}.
    \]
\end{dfn}
The singular ideal $J$ is a (closed) ideal in $C^*_r(G)$. For $a \in C^*_r(G)$, we have $a \in J$ if and only if $s\big(\osupp(a)\big)$ has empty interior in $\Gu$ (see \cite[Proposition~4.6]{BKM}, \cite[Lemma~4.1(iii)]{BGHL}).

\vspace{0.3cm}
Now assume that the groupoid $G$ is covered by countably many open bisections. We say a unit $x \in \Gu$ is \emph{Hausdorff} if for any $g \in G_x^x \setminus \{x\}$ there exists an open neighbourhood $U$ of $g$ such that $U \cap \Gu = \emptyset$. It is clear that $G$ is Hausdorff if and only if all units are Hausdorff. We let $C \subseteq \Gu$ denote the set of Hausdorff units. By \cite[Lemma~7.15]{KM21}, $C$ is dense in $\Gu$ (the Hausdorff points complement the so-called \emph{dangerous} points of \cite{KM21}). An element $a \in C^*_r(G)$ belongs to the singular ideal $J$ if and only if $s\big(\osupp(a)\big) \subseteq \Gu \setminus C$ (see \cite[Proposition~7.18]{KM21}).

\PRLsep

\subsection{Hausdorff cover}\label{prelim:HffCover}

We now introduce the Hausdorff cover groupoid. See \cite{BGHL} for an extensive study of the Hausdorff cover for non-Hausdorff \'etale groupoids. Let $G$ be a (non-Hausdorff) {\'e}tale groupoid, and let $\mathfrak{C}(G)$ be its space of closed subsets. Singleton sets in $G$ are closed, hence there is a canonical inclusion $\iota \colon G \to \mathfrak{C}(G)$ given by $\iota(g) \coloneq \{g\}$. When equipped with the Fell topology, $\mathfrak{C}(G)$ has the structure of a compact Hausdorff space (see \cite{Fell}).

\begin{dfn}[\cite{Tim}]
Let $G$ be an {\'e}tale groupoid. The \emph{Hausdorff cover} $\tiG$ is defined to be the closure of $\iota(G)$ in $\mathfrak{C}(G) \setminus \{\emptyset\}$ with respect to the Fell topology.   
\end{dfn}

Elements of $\tiG$ can be described explicitly as follows. A non-empty subset $\bm{g} \subseteq G$ belongs to $\tiG$ if and only if there exists a net $(g_\alpha)$ in $G$ whose  set of limit points is precisely $\bm{g}$, and such that every subnet of $(g_\alpha)$ has its limit points contained in (and hence equal to) $\bm{g}$. Moreover, given a net $(g_\alpha)$ in $G$, the image net $\iota(g_\alpha)$ converges to $\bm{g} \in \tiG$ in the Fell topology if and only if the conditions above are satisfied. We will use the following fact several times: any convergent net $(g_\alpha)$ in $G$ has a subnet $(g_\beta)$ such that $\iota(g_\beta)$ is convergent in $\tiG$.

Define $\tiGu \coloneq \{\bm{g} \in \tiG \colon \bm{g} \cap \Gu \neq \emptyset\}$. This set will be the unit space of $\tiG$ with the groupoid structure described below. The canonical inclusion map $\iota \colon G \hookrightarrow \tiG$ restricts to an inclusion $\Gu \hookrightarrow \tiGu$, and the image $\iota(\Gu)$ is dense in $\tiGu$. In particular, any element $\bm{x} \in \tiGu$ is contained in the closure $\overline{\Gu}$. For each $\bm{x} \in \tiGu$ there is a unique element $ \pi(\bm{x})\in \bm{x}\cap \Gu$, and the map $\pi \colon \tiGu \to \Gu$ is a proper continuous surjection (since $\iota(\Gu)$ is dense in $\tiGu$). Continuity of the groupoid operations on $G$ implies that any element $\bm{x} \in \tiGu$ is a subgroup of the isotropy group $G_{\pi(\bm{x})}^{\pi(\bm{x})}$. We will often write $\bm{x} = X$ to distinguish elements $x\in\Gu$ from $\bm{x}\in \tiGu$.

We now describe the groupoid operations on $\tiG$. With these operations, $\tiG$ becomes a Hausdorff \'etale groupoid. Note that $r(g_1) = r(g_2)$ and $s(g_1) = s(g_2)$ whenever $g_1, g_2 \in \bm{g}$ and $\bm{g} \in \tiG$. Since elements of $\tiGu$ are subgroups of isotropy groups, we have that for any $\bm{g}\in \tiG$, there are \emph{unique} elements $X, Y\in \tiGu$ such that $\bm{g} = g\cdot X = Y\cdot g$ for all $g\in\bm{g}$. Let $s(\bm{g}) \coloneq X$ and $r(\bm{g}) \coloneq Y$. Explicitly, we have
\begin{equation}\label{RangeSourceinCover}
    s(\bm{g}) = g^{-1} \cdot \bm{g} \quad\quad \text{and} \quad\quad r(\bm{g}) = \bm{g} \cdot g^{-1} \quad\quad \text{for any }  g \in \bm{g}.
\end{equation}
Define the inverse $\bm{g}^{-1} \coloneq \{g^{-1} \colon g \in \bm{g}\}$ and the product $\bm{g}\bm{h} \coloneq \{gh \colon g \in \bm{g}, h \in \bm{h}\}$ whenever $s(\bm{g}) = r(\bm{h})$. Note that
\begin{equation}\label{MultInCover}
\bm{g}\bm{h} = g \cdot \bm{h} = \bm{g} \cdot h \quad \text{ for any } g \in \bm{g}, h \in \bm{h}.
\end{equation}
\vspace{-0.1cm}

Let $f \in \cg$, and view it as a function on $\iota(G) \subseteq \tiG$ via the inclusion $\iota \colon G \to \tiG$. Then $f$ extends (uniquely) to a compactly supported continuous function $\mfi(f)$ on $\tiG$. Moreover, the function $\mfi(f) \in C_c(\tiG)$ can be described explicitly by the formula
\begin{equation}\label{Cond1}
\mfi(f)(\bm{g}) = \sum_{g \in \bm{g}} f(g)
\end{equation}
for all $\bm{g} \in \tiG$ (\cite[Lemma~3.2]{BGHL}). The embedding $\mfi \colon \cg \hookrightarrow C_c(\tiG)$ is a $^*$-homomorphism, and induces canonical embeddings $\mfi_{r} \coloneq C^*_r(\mfi) \colon C^*_r(G) \hookrightarrow C^*_r(\tiG)$ and $\mfi \coloneq C^*(\mfi) \colon C^*(G) \hookrightarrow C^*(\tiG)$ (\cite[Lemma~3.8 \& Corollary~6.8]{BGHL}). 

\vspace{0.3cm}
Now assume that the groupoid $G$ is covered by countably many open bisections. Observe that a unit $x \in \Gu$ is Hausdorff if and only if the inclusion $\iota \colon \Gu \to \tiGu$ is continuous at $x$, if and only if $\pi^{-1}\big(\{x\}\big) = \{\iota(x)\}$. Define $\tiGu_\ess \coloneq \overline{\iota(C)}$, where $C \subseteq \Gu$ is the set of Hausdorff units. This is an invariant subset of the unit space $\tiGu$ of the Hausdorff cover. We define the \emph{essential Hausdorff cover} $\tiG_\ess \coloneq \tiG|_{\tiGu_\ess} = \{\bm{g} \in \tiG \colon s(\bm{g}) \in \tiGu_\ess\}$ (\cite[Definition~4.13]{BGHL}).

\begin{dfn}[\cite{Hume}] \label{dfn:mathcalX(x)}
Let $G$ be an \'etale groupoid that is covered by countably many open bisections. For each $x\in \Gu$ define $\mathcal{X}(x) \coloneq \pi^{-1}\big(\{x\}\big)\cap\tiGu_{\ess}$.
\end{dfn}
By \cite{Hume}, $\mathcal{X}(x)$ is a conjugation-invariant collection of subgroups of $G^x_x$ that is closed in the Chabauty topology (see Section~\ref{sub:CharDens}). A subgroup $X \subseteq G^x_x$ belongs to $\mathcal{X}(x)$ if and only if there exists a net $(x_\alpha) \subset C$ of (Hausdorff) units whose set of limit points is precisely $X$, and such that every subnet of $(x_{\alpha})$ has its limit points contained in $X$.

\PRLsep

\subsection{Isotropy fibres of ideals}\label{prelim:IsoFibIdeals}

Our attention now turns to restriction maps. Let $G$ be an \'etale groupoid. Given a unit $x \in \Gu$, consider the restriction map $\eta_x \colon \cg \to \mathbb{C}[G^x_x]$, where $\mathbb{C}[G^x_x]$ denotes the group algebra of the isotropy group $G^x_x$ (thought of as finitely supported functions $G^{x}_{x}\to\mathbb{C}$). The map $\eta_x$ extends to a completely positive contraction $C^*(G) \to C^*(G^x_x)$, which we will also denote by $\eta_x$. This can be seen using induced representations of the isotropy group (see \cite[Lemma~1.2]{CN22} for example). For a (closed) ideal $I \trianglelefteq C^*(G)$, the image $\eta_x(I)$ is also a (closed) ideal in $C^*(G^x_x)$ (\cite[Lemma~2.1]{CN24}). We write $I_x \coloneq \eta_x(I)$ and, following \cite{CN24}, call $I_x$ the \emph{isotropy fibre} of $I$ at $x$. 

For $g\in G^x_x$, denote by $\delta_{g}\in\mathbb{C}[G^{x}_{x}]$ the function equal to 1 at $g$ and zero everywhere else. Each $g \in G$ induces an isomorphism of group algebras $\text{Ad}_g \colon \mathbb{C}\big[G^{s(g)}_{s(g)}\big] \to \mathbb{C}\big[G^{r(g)}_{r(g)}\big]$; $\delta_{h}\mapsto \delta_{ghg^{-1}}$ and this map extends to an isomorphism of $C^*$-algebras $\Psi_g \colon C^*\big(G^{s(g)}_{s(g)}\big) \to C^*\big(G^{r(g)}_{r(g)}\big)$. A family $(K_x)_{x \in \Gu}$ of ideals satisfying $K_x \trianglelefteq C^*(G_x^x)$ for each $x \in \Gu$ is said to be \emph{invariant} if $\Psi_g(K_{s(g)}) = K_{r(g)}$ for all $g \in G$. Given an ideal $I \trianglelefteq C^*(G)$, its isotropy fibres form an invariant family of ideals. Going in the other direction, if $\mathcal{K} = (K_x)_{x \in \Gu}$ is an invariant family of ideals, we define
\[
I(\mathcal{K}) \coloneq \{a \in C^*(G) \colon \eta_x(a^*a) \in K_x \text{ for all } x \in \Gu\}.
\]
By \cite[Corollary~2.9]{CN24}, $I(\mathcal{K})$ is an ideal in $C^*(G)$. We isolate the following observation from \cite{CN24}.

\begin{lem}\label{FamId}
    Let $\mathcal{K} = (K_x)_{x \in \Gu}$ be an invariant family of ideals, and let $A \subseteq C^*(G)$ be a $C^*$-subalgebra. Then $A \subseteq I(\mathcal{K})$ if and only if $\eta_x(A) \subseteq K_x$ for all $x \in \Gu$.
\end{lem}

\begin{proof}
    Assume that $A$ satisfies $A \subseteq I(\mathcal{K})$. Fix $x \in \Gu$. We have $\eta_x(a) \in K_x$ for all positive elements $a \in A^+$. Since $A$ is spanned by its positive elements, it follows that $\eta_x(a) \in K_x$ for all $a \in A$. Conversely, assume that $\eta_x(A) \subseteq K_x$ for all $x \in \Gu$. For any $a \in A$ we have $\eta_x(a^*a) \in K_x$ for all $x \in \Gu$, and this implies that $a \in I(\mathcal{K})$ as desired.
\end{proof}

\section{Ideals are determined by their isotropy fibres}\label{sec:Ideals}

Let $G$ be an \'etale groupoid and let $\tiG$ denote its Hausdorff cover. Let $I$ be an ideal in the full groupoid $C^*$-algebra $C^{*}(G)$. Denote by $\tilde{I}$ the ideal in $C^{*}(\tiG)$ generated by $\mfi(I)$, where $\mfi \colon C^{*}(G) \hookrightarrow C^{*}(\tiG)$ is the canonical inclusion as in Subsection \ref{prelim:HffCover}. When $I = J$, the ideal $\tilde{I}$ is not to be confused with the ideal appearing in \cite[Definition~4.14]{BGHL} and in \cite{Hume} using the same notation. 
\begin{rmk}
    The ideal $\tilde{I}$ is equal to the closed linear span of $C_{0}(\tiGu)\cdot \big(\mfi(I)\big)\cdot C_{0}(\tiGu)$ since $C^{*}(\tiG)$ is the closed linear span of $\im(\mfi) \cdot C_{0}(\tiGu)$. 
\end{rmk}
Our first observation is that the assignment $I \mapsto \tilde{I}$ is injective.

\begin{prp}\label{Ideal1to1}
    For any ideal $I$ in $C^{*}(G)$, we have $\mfi^{-1}(\tilde{I}) = I$.
\end{prp}

\begin{proof}
    Clearly we have $I\subseteq \mfi^{-1}(\tilde{I})$. To prove the reverse containment, let $\rho \colon C^{*}(G)\to B(H)$ be a $^*$-representation such that $\ker(\rho) = I$. By \cite[Lemma~6.7]{BGHL}, there exists a $^*$-representation $\tilde{\rho} \colon C^{*}(\tiG)\to B(H)$ such that $\tilde{\rho}\circ \mfi = \rho$. Therefore, $\mfi(I) = \mfi(\ker(\rho))\subseteq \ker(\tilde{\rho})$ and hence $\tilde{I}\subseteq \ker(\tilde{\rho})$. It follows that $\mfi^{-1}(\tilde{I})\subseteq \mfi^{-1}(\ker(\tilde{\rho})) = \ker(\rho) = I$, proving the proposition.
\end{proof} 

Let $X \in \tiGu$ and write $ x \coloneq \pi(X) \in \Gu$. By definition, $X$ is equal to a subgroup of the isotropy group $G^{x}_{x}$. By \eqref{RangeSourceinCover} it is easily seen that the isotropy group $\tiG^{X}_{X} \subseteq \tiG$ is equal to the quotient $N_{X}/X$, where $N_{X} = \{g\in G^{x}_{x} \colon  gXg^{-1} = X\}$ is the normaliser of $X$ in $G^{x}_{x}$. Denote by $E_{N_{X}} \colon C^*(G^x_x) \to C^*(N_{X})$ the canonical conditional expectation and by $Q_{X} \colon C^{*}(N_{X})\to C^{*}(N_{X}/X)$ the $^*$-homomorphism induced by the quotient map $N_{X} \to N_{X}/X$.

\begin{lem}
    \label{restricted_restriction}
    Let $X\in \tiGu$ and write $x \coloneq \pi(X)\in \Gu$. The following diagram commutes.
    \[
    \begin{tikzcd}
    C^{*}(G) \arrow[r, "\mfi"] \arrow[d, "\eta_{x}"] & C^{*}(\tiG) \arrow[d, "\eta_{X}"] \\
    C^{*}(G^{x}_{x})     \arrow[r, "Q_X \circ E_{N_X}"]           & C^{*}(\tiG^{X}_{X})                        
    \end{tikzcd}
    \]
\end{lem}

\begin{proof}
     Equality is readily verified on the dense $^*$-subalgebra $\cg$. The lemma then follows by continuity.
\end{proof}
 
We can now describe the isotropy fibres of the ideals $\tilde{I}$ in $C^*(\tiG)$ in terms of $I$.

\begin{prp}\label{DescribeIsoFib}
    Let $I$ be an ideal in $C^*(G)$. Let $X \in \tiGu$ and write $x \coloneq \pi(X) \in \Gu$. Then,
    \[
    \tilde{I}_X = \overline{Q_{X}\circ E_{N_{X}}(I_{x})} \trianglelefteq C^*(\tiG^X_X).
    \]
\end{prp}

\begin{proof}
 For each $X \in \tiGu$ define the ideal $K_X \coloneq \overline{Q_{X} \circ E_{N_{X}}(I_{x})} \trianglelefteq C^*(\tiG^X_X)$. Let $\bm{g} \in \tiG$, and observe that $\Psi_{\bm{g}} \circ Q_{s(\bm{g})} \circ E_{N_{s(\bm{g})}} = Q_{r(\bm{g})} \circ E_{N_{r(\bm{g})}} \circ \Psi_{g_0}$ for any $g_0 \in \bm{g}$, where $\Psi_{\bm{g}}$ and $\Psi_{g_0}$ are as defined in subsection \ref{prelim:IsoFibIdeals} (this equality can be seen using \eqref{MultInCover}). Since $\Psi_{\bm{g}}$ is an isomorphism, and since $(I_x)_{x \in \Gu}$ is an invariant family of ideals, it follows that
    \begin{align*}
    \Psi_{\bm{g}}\big(K_{s(\bm{g})}\big)
    &= \overline{\Psi_{\bm{g}}\Big(Q_{s(\bm{g})} \circ E_{N_{s(\bm{g})}}\big(I_{s(g_0)}\big)\Big)}
    = \overline{Q_{r(\bm{g})} \circ E_{N_{r(\bm{g})}}\Big(\Psi_{g_0}\big(I_{s(g_0)}\big)\Big)}\\
    &= \overline{Q_{r(\bm{g})} \circ E_{N_{r(\bm{g})}}\big(I_{r(g_0)}\big)}
    = K_{r(\bm{g})}.
    \end{align*}
    Hence, $\mathcal{K} \coloneq (K_X)_{X \in \tiGu}$ is an invariant family of ideals. 
    
    By Lemma \ref{restricted_restriction}, we have $K_X = \overline{\eta_X\big(\mfi(I)\big)}$, and therefore $K_X \subseteq \tilde{I}_X$ for all $X \in \tiGu$. By Lemma \ref{FamId}, this equality also implies that $\mfi(I) \subseteq I(\mathcal{K})$, and hence $\tilde{I} \subseteq I(\mathcal{K})$. Another application of Lemma \ref{FamId} gives $\tilde{I}_X \subseteq K_X$ for all $X \in \tiGu$, as desired.
\end{proof}

We now prove that ideals in $C^*(G)$ are determined by their isotropy fibres whenever $G$ is an amenable second-countable \'etale groupoid. This generalises \cite[Theorem~2.10]{CN24} to the setting of non-Hausdorff groupoids. Our proof utilises this theorem, which in turn uses known cases of the Effros-Hahn conjecture \cite{IW08}. Therefore, it is necessary for us to only consider amenable second-countable groupoids. The theorem does not hold in the non-amenable setting -- counterexamples are provided in \cite[Examples 2.11 \& 2.12]{CN24}.

\begin{thm}\label{IdealFib}
    Let $G$ be an amenable second-countable (non-Hausdorff) \'etale groupoid and suppose $I$ and $K$ are ideals in $C^*(G)$. Then, $I = K$ if and only if $I_{x} = K_{x}$ for all $x\in \Gu$.
\end{thm}
\begin{proof}
    The forwards implication is trivial. Assume $I_{x} = K_{x}$ for all $x\in \Gu$. Proposition \ref{DescribeIsoFib} implies that $\tilde{I}_{X} = \tilde{K}_{X}$ for all $X\in \tiGu$. Since $G$ is amenable and second-countable, so is $\tiG$ by \cite[Theorem~6.5]{BGHL}. Moreover, the groupoid $\tiG$ is Hausdorff, and therefore \cite[Theorem~2.10]{CN24} ensures that $\tilde{I} = \tilde{K}$. Finally, Proposition \ref{Ideal1to1} gives $I = \mfi^{-1}(\tilde{I}) = \mfi^{-1}(\tilde{K}) = K$.
\end{proof}

\section{Characterisation of density}\label{sub:CharDens}
In this section, we assume that $G$ is an amenable second-countable (not necessarily Hausdorff) \'etale groupoid, and therefore $C^*(G) = C^*_r(G)$ by \cite[Corollary~6.10]{BGHL}. We denote by $J$ the singular ideal in $C^*_r(G)$. Theorem \ref{IdealFib} yields the following characterisation for the density of $J \cap \cg$ in $J$.
\begin{prp}\label{DensChar}
    Let $G$ be an amenable and second-countable \'etale groupoid. Then, $J\cap \cg$ is dense in $J$ if and only if $J_{x} \cap \mathbb{C}[G^{x}_{x}]$ is dense in $J_{x}$ for all $x\in \Gu$.
\end{prp}

\begin{proof}
    Let $I \coloneq \overline{J\cap\mathscr{C}_{c}(G)}$. Since $\cg$ is dense in $C^{*}_{r}(G)$, $I$ is an ideal and $I\subseteq J$. By \cite[Theorem~5.5 \& Proposition~5.19]{Hume} we have $\eta_{x}(J\cap\mathscr{C}_{c}(G)) = J_{x}\cap\mathbb{C}[G^{x}_{x}]$. Continuity of $\eta_{x}$ then implies $J_{x}\cap\mathbb{C}[G^{x}_{x}]\subseteq \eta_{x}(I) \subseteq \overline{J_{x}\cap\mathbb{C}[G^{x}_{x}]} $. Since $\eta_{x}(I) = I_x$ is closed, it follows that $I_{x} = \overline{J_{x}\cap\mathbb{C}[G^{x}_{x}]}$. Therefore, Theorem \ref{IdealFib} shows that $I=J$ if and only if $\overline{J_{x} \cap \mathbb{C}[G^{x}_{x}]} = J_{x}$ for all $x\in \Gu$.
\end{proof}

\begin{rmk}
    The ``only if'' direction of Proposition \ref{DensChar} holds for all \'etale groupoids $G$ because $\eta_{x}(J\cap\mathscr{C}_{c}(G)) \subseteq J_{x}\cap\mathbb{C}[G^{x}_{x}]$ holds trivially ($J_{x}$ here means the isotropy fibre from the reduced groupoid $C^{*}$-algebra). It can be shown using the calculation of isotropy fibres of the singular ideal in \cite[Theorem~5.5]{Hume} that the first example in \cite{MS25} of a non-Hausdorff groupoid $G$ with $J\cap\mathscr{C}_{c}(G)$ not dense in $J$ has an isotropy fibre $J_{\epsilon}\neq \{0\}$ with $J_{\epsilon}\cap\mathbb{C}[\Gamma] = \{0\}$. Therefore, this remark may be applied to this example to provide an alternate proof of non-density.
\end{rmk}

Following \cite{Hume}, we now explain how the density question can be reduced to one about explicit ideals in group $C^{*}$-algebras. Let $\Gamma$ be an amenable discrete group. Let $\text{Sub}(\Gamma)$ denote the set of all subgroups of $\Gamma$. Consider the map
\begin{align*}
    \text{Sub}(\Gamma) &\longrightarrow \{0, 1\}^\Gamma\\
    X &\longmapsto \mathbbm{1}_X
\end{align*}
where $\mathbbm{1}_X \colon \Gamma \to \{0,1\}$ is the indicator function, and equip $\{0, 1\}^\Gamma$ with the product topology. We equip $\text{Sub}(\Gamma)$ with the topology induced by this map, known as the \emph{Chabauty topology}, making $\text{Sub}(\Gamma)$ a Stone space (this topology agrees with the Fell topology of \cite{Fell}).

\begin{sloppypar}
    For a subgroup $X \in \text{Sub}(\Gamma)$, write $\Gamma/X$ for the set of left-cosets. Let $\lambda_{\Gamma/X} \colon \Gamma \to B\big(\ell^2(\Gamma/X)\big)$ denote the associated quasi-regular representation as described in \cite{BK20}. Concretely, we have $\lambda_{\Gamma/X}(g)\delta_{hX} \coloneq \delta_{ghX}$ for $g \in \Gamma$ and $hX \in \Gamma/X$, where $\delta_{hX} \in \ell^2(\Gamma/X)$ denotes a standard basis vector. We also write $\lambda_{\Gamma/X}$ for the associated $^*$-representation on $C^*_r(\Gamma) = C^*(\Gamma)$.
\end{sloppypar}

Let $\mathcal{X} \subseteq \text{Sub}(\Gamma)$. Following \cite{Hume}, we let the ideal $J_{\Gamma,\mathcal{X}}$ denote the intersection $ \bigcap_{X\in\mathcal{X}} \ker(\lambda_{\Gamma/X})$ inside $C^{*}_{r}(\Gamma) = C^*(\Gamma)$. We will primarily be interested in the case where $\mathcal{X} \subseteq \text{Sub}(\Gamma)$ is closed in the Chabauty topology and invariant under conjugation by elements of $\Gamma$.

\begin{dfn}\label{DefnPropD}
    Let $\Gamma$ be an amenable discrete group. Define $\mathcal{D}_\Gamma$ to be the set of all closed conjugation-invariant sets of subgroups $\mathcal{X} \subseteq \text{Sub}(\Gamma)$ for which $J_{\Gamma,\mathcal{X}}\cap\mathbb{C}[\Gamma]$ is dense in $J_{\Gamma,\mathcal{X}}$. We will say $\Gamma$ satisfies the \emph{Density Property} if $\mathcal{D}_\Gamma$ contains all closed conjugation-invariant sets of subgroups.
\end{dfn}

Let $G$ be an amenable second-countable \'etale groupoid. For $x\in \Gu$, recall from Definition~\ref{dfn:mathcalX(x)} that $\mathcal{X}(x) \coloneq \pi^{-1}\big(\{x\}\big)\cap\tiGu_{\ess}$ is a closed conjugation-invariant collection of subgroups of the isotropy group $G^x_x$. The isotropy groups of $G$ are amenable, and in this case \cite[Definition~5.1 \& Theorem~5.5]{Hume} showed that $J_{x} = J_{G^{x}_{x},\mathcal{X}(x)}$ for every $x \in \Gu$. The following theorem is then immediate.

\begin{thm}\label{ReductionToIso}
    Let $G$ be an amenable and second-countable \'etale groupoid. Then $J\cap\cg$ is dense in $J$ if and only if $\mathcal{X}(x) \in \mathcal{D}_{G^x_x}$ for all $x \in \Gu$. In particular, if all the isotropy groups $G^x_x$ satisfy the Density Property, then $J\cap\cg$ is dense in $J$.
\end{thm}

\begin{rmk}
    In Theorem~\ref{ReductionToIso} it suffices to check whether $\mathcal{X}(x)\in \mathcal{D}_{G^{x}_{x}}$ for \emph{extremely dangerous points} $x \in \Gu$ (as described in \cite[Proposition~1.12]{NS23}). This is because $J_{\Gamma,\mathcal{X}} = \{0\}$ whenever the trivial group $\{e\}$ belongs to $\mathcal{X}$, and a unit $x \in \Gu$ satisfies $\{x\}\notin\mathcal{X}(x)$ if and only if it is extremely dangerous (see  \cite[Corollary~5.6]{Hume}).
\end{rmk}

It is clear that finite groups satisfy the Density Property, and hence $J \cap \cg$ is dense in $J$ for any amenable second-countable \'etale groupoid with finite isotropy groups.

\begin{question}
    Which discrete amenable groups satisfy the Density Property?
\end{question}

    Assume that there exists a (discrete) countable amenable group $\Gamma$ that fails to have the Density Property, and let $\mathcal{X} \subseteq \text{Sub}(\Gamma)$ be a closed conjugation-invariant set of subgroups for which $J_{\Gamma,\mathcal{X}}\cap\mathbb{C}[\Gamma]$ is not dense in $J_{\Gamma,\mathcal{X}}$ (i.e. $\mathcal{X} \notin \mathcal{D}_\Gamma$). In particular, the trivial subgroup $\{e\}$ does not belong to $\mathcal{X}$. By \cite[Section~6]{Hume}, there exists an amenable and second-countable \'etale groupoid $G$ and a unit element $x_0 \in \Gu$ for which $G^{x_0}_{x_0} = \Gamma$ and $\mathcal{X}(x_0) = \mathcal{X}$. It follows from Theorem \ref{ReductionToIso} that $J \cap \cg$ is not dense in $J$. Therefore, we have established the following.

\begin{cor}
\label{equivalence}
    The intersection $J \cap \cg$ is dense in $J$ for all second-countable amenable \'etale groupoids if and only if all (discrete) countable amenable groups satisfy the Density Property.
\end{cor}

\section{Groups with the Density Property}\label{sec:GrpPropD}
In this section, we study the set $\mathcal{D}_\Gamma$ for amenable discrete groups $\Gamma$, and show that certain classes of groups satisfy the Density Property (see Definition~\ref{DefnPropD}). We do this in three parts. In Subsection~\ref{subsec:reduction} we simplify the conditions one needs to check for finding elements of the set $\mathcal{D}_\Gamma$. In Subsection~\ref{subset:permanence} we prove some permanence results for the Density Property. We then use these observations in Subsection~\ref{subsec:applications} to prove Theorem~\ref{thm:bigthm}, which in particular states that all countable finite-by-nilpotent groups satisfy the Density Property. Combining with Theorem \ref{ReductionToIso}, we are able to deduce that $J \cap \cg$ is dense in the singular ideal $J$ for certain classes of groupoids.

Throughout the section we work with the reduced group $C^*$-algebra $C^*_r(\Gamma)$ since this is canonically isomorphic to the full $C^*$-algebra. The set of subgroups $\text{Sub}(\Gamma)$ will always be equipped with the Chabauty topology, and given a closed and conjugation-invariant subset $\mathcal{X} \subseteq \text{Sub}(\Gamma)$ we write $J_{\Gamma, \mathcal{X}} \coloneq \bigcap_{X \in \mathcal{X}} \ker (\lambda_{\Gamma/X}) \trianglelefteq C^*_r(\Gamma)$ (see Section \ref{sub:CharDens}). 

\vspace{0.3cm}
Let us record three preliminary lemmas which will be useful throughout the subsections. For a subgroup $\Lambda \subseteq \Gamma$, we write $I_\Lambda \colon C^*_r(\Lambda) \to C^*_r(\Gamma)$ for the canonical inclusion and $E_\Lambda \colon C^*_r(\Gamma) \to C^*_r(\Lambda)$ for the canonical conditional expectation.  By \cite[Proposition~7.4]{Hume}, $\Lambda\cap\mathcal{X} \coloneq \{\Lambda\cap X \colon X\in\mathcal{X}\}$ is a closed and conjugation-invariant set of subgroups of $\Lambda$. In the first lemma we isolate the following observations from \cite[Propositions~7.4 \& 7.7]{Hume} which describe the images of $J_{\Lambda, \Lambda \cap \mathcal{X}}$ and $J_{\Gamma, \mathcal{X}}$ under the inclusion and conditional expectation maps.

\begin{lem}\label{Inclusions}
    Let $\Gamma$ be an amenable discrete group and $\mathcal{X} \subseteq \mathrm{Sub}(\Gamma)$ be a closed and conjugation-invariant set of subgroups.
    \begin{enumerate}[label=(\roman*)]
        \item For any subgroup $\Lambda \subseteq \Gamma$ we have
        \[
        I_{\Lambda} \big(J_{\Lambda, \Lambda \cap \mathcal{X}}\big) = J_{\Gamma, \mathcal{X}}\cap I_{\Lambda}\big(C^{*}_{r}(\Lambda)\big) \quad
        \text{and}
        \quad I_{\Lambda} \big(J_{\Lambda, \Lambda \cap \mathcal{X}} \cap \mathbb{C}[\Lambda] \big) = J_{\Gamma,\mathcal{X}}\cap\mathbb{C}[\Lambda]
        \]
        \item Let $N \subseteq \Gamma$ be the (normal) subgroup generated by all $X\in\mathcal{X}$. Then,
        \[
        E_N \big(J_{\Gamma, \mathcal{X}} \big) = J_{N, \mathcal{X}}\quad 
        \text{and}
        \quad E_{N}\big(J_{\Gamma,\mathcal{X}}\cap\mathbb{C}[\Gamma]\big) = J_{N,\mathcal{X}}\cap\mathbb{C}[N].
        \] 
    \end{enumerate}
\end{lem}

\begin{proof}
    For a proof of (i) see \cite[Proposition~7.4]{Hume}. For a proof of (ii), see the proof of \cite[Proposition~7.7]{Hume} (where $\Phi$ denotes $E_{N}$).
\end{proof}

The next two lemmas concern ideals and dense $^*$-subalgebras inside general $C^*$-algebras.

\begin{lem}\label{ApproxUnit}
    Let A be a $C^*$-algebra, $\mathcal{A} \subseteq A$ a dense $^*$-subalgebra, and $I \trianglelefteq A$ an ideal. Then $I \cap \mathcal{A}$ is dense in $I$ if and only if $I \cap \mathcal{A}$ contains a bounded net $(u_{\beta})$ such that $\lim_{\beta}\|a - au_{\beta}\| = 0$ for all $a\in I$. Moreover, the $u_{\beta}$ can be chosen to be positive contractions.
\end{lem}

We will need the following consequence of this lemma: For any closed and conjugation-invariant set $\mathcal{X} \subseteq \text{Sub}(\Gamma)$, we have $\mathcal{X} \in \mathcal{D}_\Gamma$ if and only if $J_{\Gamma, \mathcal{X}} \cap \mathbb{C}[\Gamma]$ contains a net of positive contractions converging strongly to the identity in $J_{\Gamma, \mathcal{X}}$.

\begin{proof}
    Assume that $I \cap \mathcal{A}$ is dense in $I$. Since $I$ is a $C^*$-algebra, it contains an approximate unit $(u_\beta)$. Then $I \cap \mathcal{A}$ is a dense $^*$-subalgebra, so each $u_\beta$ can be approximated by a sequence of positive elements in the unit ball of $I\cap \mathcal{A}$. By a diagonalisation argument, $I\cap\mathcal{A}$ contains a net of positive contractions in $I \cap \mathcal{A}$ converging strongly to the identity in $I$. Conversely, assume that $I \cap \mathcal{A}$ contains a net of elements of norm at most $M>0$ converging strongly to the identity in  $I$. Let $a \in I$ and $\varepsilon >0$. There exist $u \in I \cap \mathcal{A}$ and $b \in \mathcal{A}$ such that $\norm{u} \le M$, $\norm{a-au} < \frac{\varepsilon}{2}$ and $\norm{a-b} < \frac{\varepsilon}{2M}$. Then $bu \in I \cap \mathcal{A}$ and
    \[
    \norm{a-bu} \le \norm{a-au} + \norm{(a-b)u } < \varepsilon.
    \]
\end{proof}

\begin{lem}
    \label{extensions}
    Let $A$ be a $C^{*}$-algebra, $\mathcal{A}\subseteq A$ a $^*$-subalgebra and $I\trianglelefteq A$ an ideal. Let $Q \colon A\to A/I$ denote the quotient map. If $Q(\mathcal{A})$ is dense in $Q(A)$ and $I\cap\mathcal{A}$ is dense in $I$, then $\mathcal{A}$ is dense in $A$.  
\end{lem}

\begin{proof}
   Let $a\in A$ and $\varepsilon > 0$. Since $Q(\mathcal{A})$ is dense in $Q(A)$, we may choose $b\in \mathcal{A}$ such that $\|Q(a)-Q(b)\| < \frac{\varepsilon}{2}$. By definition of the quotient norm, there is $\tilde{c}\in I$ such that $\|a - b - \tilde{c}\| < \frac{\varepsilon}{2}$. Since $I\cap\mathcal{A}$ is dense in $I$, there is $c\in I\cap\mathcal{A}$ with $\|c - \tilde{c}\| < \frac{\varepsilon}{2}$. Hence, $d \coloneq b+c\in \mathcal{A}$ satisfies $\|a - d\| <  \varepsilon$.
\end{proof}

\PRLsep

\subsection{Reduction Properties}\label{subsec:reduction} In this subsection we show that in order to determine whether $J_{\Gamma, \mathcal{X}}$ belongs to $\mathcal{D}_\Gamma$, it suffices to assume that $\Gamma$ is generated by the subgroups $X \in \mathcal{X}$ (Theorem~\ref{thm:reducetonsg}), and that $\mathcal{X}$ is minimal in a certain sense (Theorem~\ref{thm:reduceX}).

\vspace{0.3cm}
The following lemma is well-known to experts in group theory -- it allows us to approximate an arbitrary element $a\in C^{*}_{r}(\Gamma)$ by its restriction to finitely many cosets of a normal subgroup $N$. If $\Gamma$ is an amenable discrete group and $N$ a normal subgroup, define $a_{gN} \coloneq \delta_{g}\cdot E_{N}(\delta_{g^{-1}}\cdot a)\in C^{*}_{r}(\Gamma)$ for $a\in C^{*}_{r}(\Gamma)$ and $gN\in\Gamma/N$. Note that this is independent of the choice of representative $g$. Whenever $a_{gN} = 0$ for all but finitely many cosets $gN$, the (finite) sum $\sum_{gN}a_{gN} = a$.

\begin{lem}
\label{lem:Folner}
    Let $\Gamma$ be an amenable discrete group and $N$ a normal subgroup. For any $a\in C^{*}_{r}(\Gamma)$, there is a net $(a_{i})$ of finite linear combinations of $a_{gN}$ (for $gN\in \Gamma/N$) that converges to $a$.
\end{lem}
\begin{proof}
    By amenability of $\Gamma/N$ and \cite[Theorem~2.6.8]{BO08}, there is a net $(\tilde{\varphi}_{i})$ of finitely supported positive definite functions $\Gamma/N\to\mathbb{C}$ converging pointwise to $1$. Let $q_N \colon \Gamma \to \Gamma/N$ denote the quotient map, and define the functions $\varphi_{i} \coloneq \tilde{\varphi}_{i}\circ q_N$. The $\varphi_i$ are positive definite, and converge pointwise to $1$ on $\Gamma$. Without loss of generality we may assume $\varphi_{i}(e) = 1$ for all $i$. Then, by \cite[Theorem~2.5.11(4)]{BO08}, the multiplier maps $m_{i} \colon C^{*}_{r}(\Gamma)\to C^{*}_{r}(\Gamma)$ defined for $a = \sum_{g\in\Gamma}a_{g}\delta_{g}\in\mathbb{C}[\Gamma]$ as $m_{i}(a) = \sum_{g\in\Gamma} \varphi_{i}(g) a_g \delta_{g}$ are unital and completely positive, and are therefore contractions. 
    
    For $a \in C^*_r(\Gamma)$, observe that $\big(m_i(a)\big)_{gN} = \tilde{\varphi}_{i}(gN)a_{gN}$ for each $gN\in\Gamma/N$, and hence $m_i(a)$ is a finite linear combination of $a_{gN}$. Since $(\varphi_{i})$ converges pointwise to $1$, $m_i(a)$ converges to $a$ whenever $a\in\mathbb{C}[\Gamma]$. Since the $m_{i}$ are contractions, $m_i(a)$ converges to $a$ for all $a\in C^{*}_{r}(\Gamma)$. Defining $a_i \coloneq m_i(a)$, the net $(a_{i})$ satisfies the conclusion of the lemma.
\end{proof}

We now prove our first reduction property.

\begin{thm}
\label{thm:reducetonsg}
    Let $\Gamma$ be an amenable discrete group and $\mathcal{X} \subseteq \mathrm{Sub}(\Gamma)$ a closed and conjugation-invariant set of subgroups. Let $\Lambda \subseteq \Gamma$ be a subgroup such that $X \subseteq \Lambda$ for all $X \in \mathcal{X}$. Then, $\mathcal{X}\in\mathcal{D}_{\Gamma}$ if and only if $\mathcal{X}\in \mathcal{D}_{\Lambda}$.
\end{thm}

\begin{proof}
    Let $N \subseteq \Gamma$ be the (normal) subgroup generated by all $X\in\mathcal{X}$. It suffices to prove the theorem in the case $\Lambda = N$, since $N$ is independent of the ambient group.
    
    If $J_{\Gamma,\mathcal{X}}\cap\mathbb{C}[\Gamma]$ is dense in $J_{\Gamma,\mathcal{X}}$, then continuity of $E_{N}$ and Lemma \ref{Inclusions} (ii) imply that $J_{N,\mathcal{X}}\cap\mathbb{C}[N] = E_{N}(J_{\Gamma,\mathcal{X}}\cap\mathbb{C}[\Gamma])$ is dense in $J_{N,\mathcal{X}} = E_{N}(J_{\Gamma,\mathcal{X}})$. This proves the ``only if'' direction.

    Conversely, assume that $J_{N, \mathcal{X}}\cap\mathbb{C}[N]$ is dense in $J_{N,\mathcal{X}}$. By Lemma \ref{ApproxUnit} there exists a net of contractions $(u_{\beta})\subseteq J_{N,\mathcal{X}}\cap\mathbb{C}[N]$ converging strongly to the identity in $J_{N,\mathcal{X}}$. Since $J_{\Gamma,\mathcal{X}}\cap I_N\big( C^{*}_{r}(N) \big) = I_N \big(J_{N,\mathcal{X}} \big)$ by Lemma~\ref{Inclusions} (i), we can regard $(u_{\beta})$ as a net in $J_{\Gamma,\mathcal{X}}\cap\mathbb{C}[\Gamma]$.
    
    Let $a\in J_{\Gamma,\mathcal{X}}$, and let $(a_{i})$ be as in Lemma \ref{lem:Folner}. By Lemma \ref{Inclusions}, $a_{gN}\in J_{\Gamma,\mathcal{X}}$ for all $gN\in \Gamma/N$. Since $a_{i}$ is a finite linear combination of $a_{gN}$, it follows that $a_{i}\in J_{\Gamma,\mathcal{X}}$ for all $i$. Moreover, since $\delta_{g^{-1}} \cdot a_{gN}\in I_N\big(J_{N,\mathcal{X}}\big)$, we have $\lim_{\beta}\|a_{gN} \cdot u_{\beta} - a_{gN}\| = 0$ for all $gN\in\Gamma/N$ and therefore $\lim_{\beta}\|a_{i} \cdot u_{\beta} - a_{i}\| = 0$ for all $i$. Since $(a_{i})$ converges to $a$ and $\|u_{\beta}\|\leq 1$ for all $\beta$, it follows from the above that $\lim_{\beta}\|a \cdot u_{\beta} - a\| = 0$. Lemma \ref{ApproxUnit} then implies that $J_{\Gamma, \mathcal{X}}\cap\mathbb{C}[\Gamma]$ is dense in $J_{\Gamma,\mathcal{X}}$, completing the proof.
\end{proof}

The following corollary will be used in the proofs of Theorems \ref{thm:quotients} and \ref{thm:bigthm}.

\begin{cor}
\label{cor:nsg}
    If $\Gamma$ is an amenable discrete group and $N$ a normal subgroup, then $\{N\}\in\mathcal{D}_{\Gamma}$.
\end{cor}
\begin{proof}
By Theorem \ref{thm:reducetonsg}, we can assume that $N = \Gamma$. Note that the ideal $J_{\Gamma, \mathcal{X}} = J_{\Gamma,\{\Gamma\}}$ is equal to the kernel of the trivial representation $\epsilon \colon C^{*}_{r}(\Gamma)\to \mathbb{C}$. Let $(F_{i})$ be a F{\o}lner net for $\Gamma$ and set $v_{i} = \frac{1}{|F_{i}|}\sum_{g\in F_{i}}\delta_{g}$. For any $g\in \Gamma$, we have $\lim_{i}\|\delta_{g}\cdot v_{i} - v_{i}\| = 0$, hence (by linearity and density) $\lim_{i}\|a \cdot v_{i} - \epsilon(a)v_{i}\| = 0$ for all $a\in C^{*}_{r}(\Gamma)$. By Kesten's criteria (\cite[Theorem~2.6.8(8)]{BO08}) we have $\|v_{i}\| = 1$ for all $i$, so the above limit implies $\lim_{i}\|a \cdot v_{i}\| = |\epsilon(a)|$ for all $a\in C^{*}_{r}(\Gamma)$. Set $u_{i} = \delta_{e} - v_{i}$ for all $i$. Then $(u_i)$ is a net of elements in $J_{\Gamma, \{\Gamma\}} \cap \mathbb{C}[\Gamma]$ of norm at most 2 satisfying $\lim_{i}\|a \cdot u_{i} - a\| = \lim_{i} \|a \cdot v_{i}\| = |\epsilon(a)| = 0$ for any $a\in J_{\Gamma, \{\Gamma\}}$. It follows from Lemma \ref{ApproxUnit} that $J_{\Gamma, \{\Gamma\}}\cap\mathbb{C}[\Gamma]$ is dense in $J_{\Gamma,\{\Gamma\}}$, as desired.
\end{proof}

In order to determine whether $\mathcal{X}$ belongs to $\mathcal{D}_\Gamma$, Theorem \ref{thm:reducetonsg} allows us to reduce the size of the ambient group $\Gamma$. The next result allows us to reduce the size of $\mathcal{X}$, and assume that it is minimal in a certain sense.

For any discrete group $\Gamma$ and closed conjugation-invariant set $\mathcal{X} \subseteq \text{Sub}(\Gamma)$, let $\mathcal{X}_{\text{min}} \subseteq \mathcal{X}$ be the collection of subgroups $X\in\mathcal{X}$ for which $X'\subseteq X$ and $X'\in\mathcal{X}$ implies $X = X'$. Then $\mathcal{X}_{\text{min}}$ is conjugation-invariant but not necessarily closed.

\begin{thm}
\label{thm:reduceX}
    Let $\Gamma$ be a discrete group and $\mathcal{X} \subseteq \mathrm{Sub}(\Gamma)$ a closed and conjugation-invariant set of subgroups. For every $X\in \mathcal{X}$, there is $Y\in\mathcal{X}_{\mathrm{min}}$ such that $Y\subseteq X$. Consequently, for $\Gamma$ amenable, we have $J_{\Gamma,\mathcal{X}} = J_{\Gamma,\overline{\mathcal{X}}_\mathrm{min}}$, so $\mathcal{X}\in\mathcal{D}_{\Gamma}$ if and only if $\overline{\mathcal{X}}_{\mathrm{min}}\in \mathcal{D}_{\Gamma}$.
\end{thm}
\begin{proof}
     Fix $X\in \mathcal{X}$. Let $\mathcal{Z}_{X} = \{Y\in \mathcal{X} \colon Y\subseteq X\}$ and (partially) order $\mathcal{Z}_{X}$ by inclusion. Let $\mathcal{C}\subseteq \mathcal{Z}_{X}$ be a chain in $\mathcal{Z}_{X}$, and view $\mathcal{C}$ as a net ordered by inclusion. Define $Y_* \coloneq \bigcap_{Y \in \mathcal{C}} Y$. The net $\mathcal{C}$ converges to $Y_*$ by definition of the Chabauty topology, so $Y_*$ belongs to $\mathcal{X}$ and hence to $\mathcal{Z}_{X}$. Therefore, every chain in $\mathcal{Z}_{X}$ has a lower bound. Zorn's lemma implies that $\mathcal{Z}_{X}$ contains a minimal element $Y\in\mathcal{\mathcal{X}}$.
     This proves the first part of the theorem.

     Now assume that $\Gamma$ is amenable. If $X_{1}\subseteq X_{2}$ are subgroups of $\Gamma$, then $\ker(\lambda_{\Gamma/X_{1}})\subseteq \ker(\lambda_{\Gamma/X_{2}})$ (see \cite[Appendix~E~and~F]{BHV07}). Therefore, $J_{\Gamma,\mathcal{X}} =\bigcap_{X\in\mathcal{X}}\ker(\lambda_{\Gamma/X}) = \bigcap_{X\in \mathcal{X}_{\text{min}}}\ker(\lambda_{\Gamma/X})$, and by \cite[Proposition~3.3]{BK20} we have $\bigcap_{X\in \mathcal{X}_{\text{min}}}\ker(\lambda_{\Gamma/X}) = \bigcap_{X\in \overline{\mathcal{X}}_{\text{min}}}\ker(\lambda_{\Gamma/X})$. This completes the proof.
\end{proof}

\PRLsep

\subsection{Permanence Properties}\label{subset:permanence} In this subsection we show that the set $\mathcal{D}_\Gamma$ is preserved under quotients (Theorem \ref{thm:quotients}) and increasing unions (Proposition \ref{prp:induc}) of the ambient group. As a consequence, the Density Property is preserved under taking quotients and inductive limits. We also show that $\mathcal{D}_\Gamma$ contains any closed collection of \emph{normal} subgroups (Corollary~\ref{cor:arbnsg}).

Firstly, we prove that $\mathcal{D}_{\Gamma}$ is closed under finite unions.
\begin{prp}
\label{prp:finiteunion}
    Let $\Gamma$ be an amenable discrete group and suppose $\mathcal{X}_{1}, \dots, \mathcal{X}_n \subseteq \mathrm{Sub}(\Gamma)$ are closed and conjugation-invariant sets of subgroups belonging to $\mathcal{D}_{\Gamma}$. Then the union $\mathcal{X} \coloneq \bigcup_{i =1}^ n \mathcal{X}_{i}$ also belongs to $\mathcal{D}_{\Gamma}$.
\end{prp}

\begin{proof}
    By definition, we have $J_{\Gamma,\mathcal{X}} = \bigcap^{n}_{i=1}J_{\Gamma, \mathcal{X}_{i}}$. For each $i = 1, \dots, n$ there exists, by Lemma \ref{ApproxUnit}, a net $(u_{i,\beta})$ of positive contractions in $J_{\Gamma,\mathcal{X}_{i}}\cap\mathbb{C}[\Gamma]$ converging strongly to the identity in $J_{\Gamma, \mathcal{X}_{i}}$, and hence in $J_{\Gamma,\mathcal{X}} \subseteq J_{\Gamma, \mathcal{X}_{i}}$. Therefore, by a diagonalisation argument, the net $(u_{\beta_{1},...,\beta_{n}}) \coloneq (u_{1,\beta_{1}}\cdot ...\cdot u_{n,\beta_{n}})$ (equipped with the lexicographic partial ordering on $(\beta_{1},...,\beta_{n})$) has a subnet $(u_{\gamma})$ converging strongly to the identity in $J_{\Gamma,\mathcal{X}}$. The net $(u_\gamma)$ lies in $J_{\Gamma,\mathcal{X}} \cap \mathbb{C}[\Gamma]$, so $\mathcal{X}\in\mathcal{D}_{\Gamma}$ by Lemma \ref{ApproxUnit}.
\end{proof}

For an amenable discrete group $\Gamma$ and a normal subgroup $N$, let $Q_{N}:C^{*}_{r}(\Gamma)\to C^{*}_{r}(\Gamma/N)$ be the $^*$-homomorphism induced by the quotient map $q_{N}:\Gamma\to \Gamma/N$. The map 
\begin{align*}
    q^{*}_{N} \colon \text{Sub}(\Gamma/N) &\rightarrow \text{Sub}(\Gamma);
    \text{    } Y \mapsto q_{N}^{-1}(Y),
\end{align*}
is a continuous embedding satisfying $\text{Conj}_{g}\circ q^{*}_{N} = q^{*}_{N}\circ\text{Conj}_{q_{N}(g)}$ for all $g\in\Gamma$, where $\text{Conj}_{g}$ denotes conjugation by $g$. Therefore, if $\mathcal{Y}$ is a closed and conjugation-invariant set of subgroups of $\Gamma/N$, then $q^{*}_{N}(\mathcal{Y})$ is a closed and conjugation-invariant set of subgroups of $\Gamma$.

\begin{thm}
\label{thm:quotients}
    Let $\Gamma$ be an amenable discrete group and $N$ a normal subgroup. If $\mathcal{Y}\subseteq \mathrm{Sub}(\Gamma/N)$ is a closed and conjugation-invariant set of subgroups of $\Gamma/N$, then 
    \begin{equation}
    \label{LiftIdeal}
         J_{\Gamma, q_{N}^{*}(\mathcal{Y})} = Q_{N}^{-1}(J_{\Gamma/N,\mathcal{Y}}) \quad \text{and} \quad Q_{N}\big(J_{\Gamma, q_{N}^{*}(\mathcal{Y})} \cap \mathbb{C}[\Gamma] \big) = J_{\Gamma/N,\mathcal{Y}} \cap \mathbb{C}[\Gamma/N].
    \end{equation}
Moreover, $q^{*}_{N}(\mathcal{Y})\in\mathcal{D}_{\Gamma}$ if and only if $\mathcal{Y}\in\mathcal{D}_{\Gamma/N}$. Consequently, if $\Gamma$ satisfies the Density Property, then so does $\Gamma/N$.
\end{thm}

\begin{proof}
    For $Y \in \mathcal{Y}$, it is easy to see that $\lambda_{\Gamma/q_{N}^{-1}(Y)} = \lambda_{(\Gamma/N)/Y}\circ Q_{N}$ and therefore $\ker\big(\lambda_{\Gamma/q_{N}^{-1}(Y)}\big) = Q_{N}^{-1}\big(\ker(\lambda_{(\Gamma/N)/Y})\big)$. Hence,
    \[
    J_{\Gamma, q^{*}_{N}(\mathcal{Y})} = \bigcap_{Y \in \mathcal{Y}} \ker\big(\lambda_{\Gamma/q_{N}^{-1}(Y)}\big) = Q^{-1}_{N}\big(\bigcap_{Y \in \mathcal{Y}}\ker(\lambda_{(\Gamma/N)/Y})\big) = Q^{-1}_{N}(J_{\Gamma/N,\mathcal{Y}}).
    \]
    This implies $Q_{N}\big(J_{\Gamma, q_{N}^{*}(\mathcal{Y})} \cap \mathbb{C}[\Gamma] \big) = J_{\Gamma/N,\mathcal{Y}}\cap\mathbb{C}[\Gamma/N]$ since $ J_{\Gamma/N,\mathcal{Y}}\cap\mathbb{C}[\Gamma/N]\subseteq \mathbb{C}[\Gamma/N] = Q_{N}\big(\mathbb{C}[\Gamma]\big)$, proving \eqref{LiftIdeal}.
    
    Next, suppose $q_{N}^{*}(\mathcal{Y})\in\mathcal{D}_{\Gamma}$, so that $J_{\Gamma,q^{*}_{N}(\mathcal{Y})} \cap \mathbb{C}[\Gamma]$ is dense in $J_{\Gamma,q^{*}_{N}(\mathcal{Y})}$. It follows from \eqref{LiftIdeal} and continuity of $Q_N$ that $J_{\Gamma/N, \mathcal{Y}}\cap \mathbb{C}[\Gamma/N]$ is dense in $J_{\Gamma/N, \mathcal{Y}}$. Therefore $\mathcal{Y}\in\mathcal{D}_{\Gamma/N}$, as required.

    Lastly, we prove the converse -- assume $\mathcal{Y}\in\mathcal{D}_{\Gamma/N}$. Set $A = J_{\Gamma, q^{*}_{N}(\mathcal{Y})}$, $\mathcal{A} = A\cap\mathbb{C}[\Gamma]$ and $I = J_{\Gamma, \{N\}}$. Each $X\in q_{N}^{*}(\mathcal{Y})$ contains the normal subgroup $N$, hence $I\subseteq A$. Since $I = \ker(Q_{N})$, we can naturally identify $Q_N|_A$ with the quotient map $A \to A/I$. By \eqref{LiftIdeal} and the assumption $\mathcal{Y}\in\mathcal{D}_{\Gamma/N}$, we see that $Q_N(\mathcal{A})$ is dense in $Q_N(A) = J_{\Gamma/N,\mathcal{Y}}$. Also, $I\cap\mathcal{A} = J_{\Gamma, \{N\}}\cap\mathbb{C}[\Gamma]$ is dense in $I$ by Corollary \ref{cor:nsg}. Therefore, Lemma \ref{extensions} implies that $\mathcal{A}$ is dense in $A$, proving that $q_{N}^{*}(\mathcal{Y})\in \mathcal{D}_{\Gamma}$.
\end{proof}

We now show the Density Property is preserved under taking countable increasing unions.

\begin{prp}
\label{prp:induc}
    Suppose that an amenable discrete group $\Gamma$ is the countable increasing union of subgroups $\Lambda_1 \subseteq \Lambda_2 \subseteq \dots \subseteq \Gamma$. If $\mathcal{X} \subseteq \mathrm{Sub}(\Gamma)$ is closed, conjugation-invariant, and such that $\Lambda_n\cap\mathcal{X}\in\mathcal{D}_{\Lambda_n}$ for all $n\in\mathbb{N}$, then $\mathcal{X}\in\mathcal{D}_{\Gamma}$.

    In particular, if all $\Lambda_n$ have the Density Property, then so does $\Gamma$.
\end{prp}
\begin{proof}
    For each $n\in\mathbb{N}$, write $\mathcal{X}_{n} \coloneq \Lambda_n\cap\mathcal{X}$. Since $\mathcal{X}_n \in \mathcal{D}_{\Lambda_{n}}$ for every $n \in \mathbb{N}$, Lemma~\ref{Inclusions} (i) implies that $J_{\Gamma,\mathcal{X}}\cap\mathbb{C}[\Lambda_n]$ is dense in $J_{\Gamma,\mathcal{X}}\cap I_{\Lambda_n} \big(C^{*}_{r}(\Lambda_n) \big)$ for every $n \in \mathbb{N}$. Hence, $J_{\Gamma,\mathcal{X}}\cap\mathbb{C}[\Gamma] = \bigcup_{n}J_{\Gamma,\mathcal{X}}\cap\mathbb{C}[\Lambda_n]$ is dense in $\bigcup_{n}J_{\Gamma,\mathcal{X}}\cap I_{\Lambda_n} \big(C^{*}_{r}(\Lambda_n)\big)$.
    The countable increasing union $\bigcup_{n=1}^\infty I_{\Lambda_n}\big(C^{*}_{r}(\Lambda_n) \big)$ is dense in $C^{*}_{r}(\Gamma)$, so $\bigcup_{n}J_{\Gamma,\mathcal{X}}\cap I_{\Lambda_n} \big(C^{*}_{r}(\Lambda_n) \big)$ is dense in $J_{\Gamma,\mathcal{X}}$ by \cite[Lemma~III.4.1]{Davidson}, completing the proof.
\end{proof}

As an application of Theorem \ref{thm:quotients} and Proposition \ref{prp:induc}, the Density Property is preserved under inductive limits of groups.

\begin{cor}
    If $(\varphi_{n}:\Gamma_{n}\to \Gamma_{n+1})_{n}$ is a sequence of homomorphisms between groups with the Density Property, then the inductive limit $\Gamma = \lim_{n}(\varphi_{n}, \Gamma_{n})$ has the Density Property.
\end{cor}
\begin{proof}
    Let $(\mu_{n}:\Gamma_{n}\to \Gamma)_{n}$ be the limit homomorphisms. By Theorem \ref{thm:quotients}, $\mu_{n}(\Gamma_{n})$ has the Density Property for every $n\in\mathbb{N}$. Since the subgroups of this form increase to $\Gamma$, Proposition \ref{prp:induc} implies $\Gamma$ has the Density Property.
\end{proof}

Let $\mathcal{X}$ be a closed and conjugation-invariant set of subgroups of $\Gamma$. Proposition~\ref{prp:induc} shows that we can approximate the ideal $J_{\Gamma,\mathcal{X}}$ through approximations of $\Gamma$. We finish the subsection by proving that we can also approximate $J_{\Gamma,\mathcal{X}}$ through approximations of $\mathcal{X}$.

Given a finite subset $F\subseteq\Gamma$ and $X\in\mathcal{X}$, let $X_{F} \subseteq X$ be the subgroup generated by the set $\bigcup_{g\in \Gamma}X\cap gFg^{-1}$. Define $\mathcal{X}_{F} \coloneq \{X_{F} \colon X\in\mathcal{X}\}$, which is conjugation-invariant, and let $J_{\Gamma,\mathcal{X}_{F}} = \bigcap_{X_{F}\in\mathcal{X}_{F}}\ker(\lambda_{\Gamma/X_{F}})$. Notice that $J_{\Gamma, \mathcal{X}_{F}}\subseteq J_{\Gamma, \mathcal{X}_{K}}$ whenever $F\subseteq K$ (see \cite[Appendix~E~and~F]{BHV07}) and $J_{\Gamma,\mathcal{X}_{F}} = J_{\Gamma, \overline{\mathcal{X}_{F}}}$ (by \cite[Proposition~3.3]{BK20}).

\begin{prp}
    \label{prp:inducsg}
    Let $\Gamma$ be a countable discrete amenable group and $\mathcal{X} \subseteq \text{Sub}(\Gamma)$ a closed and conjugation-invariant set of subgroups. Let $F_{n}\subseteq \Gamma$ be a sequence of finite sets that increase to $\Gamma$. Then, $J_{\Gamma,\mathcal{X}} = \overline{\bigcup_{n}J_{\Gamma,\mathcal{X}_{F_{n}}}}$.
\end{prp}
\begin{proof}
    For each $n\in\mathbb{N}$, define $\mathcal{X}_{n} = \overline{\bigcup_{m\geq n}\mathcal{X}_{F_{m}}}$. We have $\mathcal{X}_{n+1}\subseteq \mathcal{X}_{n}$ for all $n\in\mathbb{N}$, so that $J_{\Gamma, \mathcal{X}_{n}}\subseteq J_{\Gamma,\mathcal{X}_{n+1}}$. Set $\mathcal{X}_{\infty} = \bigcap_{n}\mathcal{X}_{n}$. We first show $J_{\Gamma,\mathcal{X}_{\infty}} = \overline{\bigcup_{n}J_{\Gamma, \mathcal{X}_{n}}}$ using that $\mathcal{X}_\infty$ and $\mathcal{X}_1, \mathcal{X}_2, \dots$ are closed conjugation-invariant collections of subgroups.
    
    For $n\in\mathbb{N}\cup\{\infty\}$, let $\Gamma\cdot\mathcal{X}_{n}$ be the \'etale groupoid of cosets of $\mathcal{X}_{n}$ (as described in \cite[Section~5.1]{Hume}) and denote by $q_{n}:C^{*}_{r}(\Gamma)\to C^{*}_{r}(\Gamma\cdot\mathcal{X}_{n})$ the canonical $^*$-homomorphism $\delta_{g}\mapsto \delta_{g\mathcal{X}_{n}}$, $g\in \Gamma$. As noted in \cite{Hume}, we have $\ker(q_{n}) = J_{\Gamma, \mathcal{X}_{n}}.$ For $n_{1}, n_{2}\in\mathbb{N}\cup\{\infty\}$ such that $n_{1}\leq n_{2}$, $\mathcal{X}_{n_2}$ is a closed $\Gamma\cdot\mathcal{X}_{n_1}$ invariant subset of $(\Gamma\cdot\mathcal{X}_{n_1})^{0} = \mathcal{X}_{n_1}$, so the restriction maps $q_{n_{2}, n_{1}}:C_{r}^{*}(\Gamma\cdot\mathcal{X}_{n_{1}})\to C_{r}^{*}(\Gamma\cdot\mathcal{X}_{n_{2}})$ are $^*$-homomorphisms. Moreover, since $\mathcal{X}_{n}$ decrease to $\mathcal{X}_{\infty}$, the maps $(q_{\infty,n})_{n\in\mathbb{N}}$ form an inductive limit for the inductive sequence $(q_{n+1,n})_{n\in\mathbb{N}}$. Therefore, by \cite[Proposition~6.2.4]{RL00} and the fact that $q_{m} = q_{m,n}\circ q_{n}$ for all $m, n\in\mathbb{N}\cup\{\infty\}$ with $m\geq n$, we have $a\in J_{\Gamma,\mathcal{X}_{\infty}}$ if and only if $\lim_{n\to\infty}\|q_{n}(a)\| = 0$. Hence, for every $\varepsilon > 0$, we can choose $n\in\mathbb{N}$ such that $\|q_{n}(a)\|\leq \frac{\varepsilon}{2}$. Also, since $\ker(q_n) = J_{\Gamma, \mathcal{X}_n}$ (and since $q_{n}$ is a quotient map) there exists $b\in J_{\Gamma,\mathcal{X}_{n}}$ such that $\big|\|a - b\| - \|q_{n}(a)\|\big |\leq \frac{\varepsilon}{2}$. Therefore, $\|a - b\|\leq \varepsilon$. This proves $J_{\Gamma,\mathcal{X}_{\infty}} = \overline{\bigcup_{n}J_{\Gamma, \mathcal{X}_{n}}}$.

    Next, let us show that $J_{\Gamma, \mathcal{X}_{n}} = J_{\Gamma, \mathcal{X}_{F_{n}}}$ for all $n\in\mathbb{N}$. By \cite[Proposition~3.3]{BK20}, we have $J_{\Gamma, \mathcal{X}_{n}} = J_{\Gamma, \bigcup_{m\geq n} \mathcal{X}_{F_{m}}}$. Since $X_{F_{n}}\subseteq X_{F_{m}}$ for all $X\in\mathcal{X}$ and $m\geq n$, it follows from \cite[Appendix~E~and~F]{BHV07} that $J_{\Gamma, \bigcup_{m\geq n} \mathcal{X}_{F_{m}}} = J_{\Gamma, \mathcal{X}_{F_{n}}}$.
    
    Finally, we prove that $\mathcal{X}_\infty = \mathcal{X}$. For fixed $X\in\mathcal{X}$, the sequence $(X_{F_{n}})$ converges to $X$ in the Chabauty topology, so that $\mathcal{X}\subseteq \mathcal{X}_{\infty}$. For the other inclusion, let $Y\in \mathcal{X}_{\infty}$. Since $\mathcal{X}_{\infty} = \bigcap_{n}\overline{\bigcup_{m\geq n}\mathcal{X}_{F_{m}}}$, there are sequences $X_k \in \mathcal{X}$ and $n_k \to \infty$ such that the sequence $((X_k)_{F_{n_k}})_k$ converges to $Y$. By passing to a subsequence if necessary, we may assume that $(X_k)_k$ converges to some $X\in\mathcal{X}$. Since $(X_k)_{F_{n_k}} \subseteq X_k$ for all $k$, it is clear that $Y \subseteq X$. On the other hand, for every $x\in X$ we have $x\in X_k$ eventually. Since the $F_{n_k}$ are increasing to $\Gamma$, we have $x\in F_{n_{k}}$ eventually. Combining these observations gives that $x\in (X_k)_{F_{n_k}}$ eventually, so $x\in Y$ and hence $X \subseteq Y$. This completes the proof of the proposition.
\end{proof}

We have the following immediate corollary.

\begin{cor}
    \label{cor:arbnsg}
    Let $\Gamma$ be a countable discrete amenable group and $\mathcal{X}$ a closed collection of normal subgroups. Then, $\mathcal{X}\in\mathcal{D}_{\Gamma}$.
\end{cor}
\begin{proof}
    Let $F$ be a finite subset of $\Gamma$. Since $X\in\mathcal{X}$ is normal, we have $\bigcup_{g\in \Gamma} X\cap gFg^{-1} = \bigcup_{g\in\Gamma} g(X\cap F)g^{-1}$, so that there are only finitely many distinct $X_{F}$, each being normal in $\Gamma$. By Corollary \ref{cor:nsg} and Proposition \ref{prp:finiteunion}, we have $\mathcal{X}_{F}\in\mathcal{D}_{\Gamma}$. Now, choosing a sequence ($F_{n}$) of finite subsets increasing to $\Gamma$, we have $J_{\Gamma, \mathcal{X}} = \overline{\bigcup_{n}J_{\Gamma,\mathcal{X}_{F_n}}}$ by Proposition~\ref{prp:inducsg}. Since  $\mathcal{X}_{F_{n}}\in\mathcal{D}_{\Gamma}$ for all $n\in\mathbb{N}$, it follows that $\mathcal{X}\in\mathcal{D}_{\Gamma}$.
\end{proof}

\PRLsep

\subsection{Applications}\label{subsec:applications}

We will now show that \emph{finite-by-nilpotent} groups satisfy the Density Property. We also show that any closed conjugation-invariant collection $\mathcal{X}$ of finite (or co-finite) subgroups of an amenable group $\Gamma$ belongs to $\mathcal{D}_\Gamma$. Let us establish a key lemma that unifies the cases.

\begin{lem}
\label{lem:finitemin}
    Let $\Gamma$ be a discrete group and $\mathcal{X} \subseteq \mathrm{Sub}(\Gamma)$ a closed and conjugation-invariant set of subgroups. If every $X\in \mathcal{X}_\mathrm{min}$ is finitely generated, then $\mathcal{X}_{\mathrm{min}}$ is finite.
\end{lem}
\begin{proof}
    For each $X\in \mathcal{X}_\text{min}$, note that $\mathcal{Z}(X) \coloneq \{X'\in\mathcal{X} \colon X\subseteq X'\} = \bigcap^{n}_{i=1}\{X'\in \mathcal{X} \colon x_{i}\in X'\}$, where $\{x_{1},...,x_{n}\}$ is a finite generating set for $X$. Therefore, $\mathcal{Z}(X)$ is open in $\mathcal{X}$ with respect to the Chabauty topology, and hence $\{\mathcal{Z}(X) \colon X \in \mathcal{X}_{\text{min}}\}$ forms an open cover of $\mathcal{X}$ by Theorem~\ref{thm:reduceX}. Since $\mathcal{Z}(X)\cap \mathcal{X}_{\text{min}} = \{X\}$ for each $X\in\mathcal{X}_{\text{min}}$, the cover $\{\mathcal{Z}(X) \colon X \in \mathcal{X}_{\text{min}}\}$ admits no proper subcover, so compactness of $\mathcal{X}$ ensures that $\mathcal{X}_{\text{min}}$ is finite.
\end{proof}

Recall that a group $\Gamma$ is \emph{nilpotent} if there is a finite sequence $\{e\} = \Lambda_{0} \trianglelefteq \Lambda_{1} \trianglelefteq \cdots \trianglelefteq \Lambda_{n} = \Gamma$ of subgroups that are normal in $\Gamma$ such that $\Lambda_{i+1}/\Lambda_{i}$ is contained in the centre $Z(\Gamma/\Lambda_{i})$ of $\Gamma/\Lambda_{i}$ for all $i = 0, \dots, n-1$. We say $\Gamma$ is \emph{finite-by-nilpotent} if there is a finite normal subgroup $N\trianglelefteq\Gamma$ such that $\Gamma/N$ is nilpotent.

A subgroup $X\subseteq \Gamma$ is called \emph{subnormal} if there is a sequence $X = \Lambda_{0}\trianglelefteq \Lambda_{1}\trianglelefteq\cdots\trianglelefteq\Lambda_{k} = \Gamma$ such that $\Lambda_i$ is normal in $\Lambda_{i+1}$ for all $i < k$. Such a sequence is called a \emph{subnormal series}. In a nilpotent group, every subgroup is subnormal. Therefore, in a finite-by-nilpotent group, every subgroup $X$ is contained as a finite index subgroup inside some subnormal subgroup $\Lambda$ (\cite[Section~6.3]{LS87}).

\vspace{0.3cm}
We record a known lemma about finite-by-nilpotent groups that immediately justifies the consideration of Lemma \ref{lem:finitemin}.

\begin{lem}
\label{lem:fbnlemma}
    Let $\Gamma$ be a finitely generated finite-by-nilpotent group. Then every subgroup of $\Gamma$ is finitely generated.
\end{lem}

\begin{proof}
    Every finitely generated finite-by-nilpotent group is finite-by-polycyclic, and the property in the lemma holds for such groups (\cite{H54}).
\end{proof}

We now see our result for finite-by-nilpotent groups and collections of finite (or co-finite) subgroups as a consequence of all our reduction and permanence properties  (Theorem \ref{thm:reducetonsg}, Theorem \ref{thm:reduceX}, Theorem \ref{thm:quotients}, Corollary \ref{cor:nsg}, Proposition \ref{prp:finiteunion}) together with the two previous lemmas.

\begin{thm}
\label{thm:bigthm}
Let $\Gamma$ be an amenable discrete group and $\mathcal{X} \subseteq \mathrm{Sub}(\Gamma)$ a closed and conjugation-invariant set of subgroups. If either
    \begin{enumerate}[label=(\Roman*)] 
        \item every $X\in\mathcal{X}$ is finite, 
        \item every $X\in\mathcal{X}$ is co-finite, or
        \item $\Gamma$ is countable and finite-by-nilpotent.
    \end{enumerate}
then $\mathcal{X}\in\mathcal{D}_{\Gamma}$.
\end{thm}
\begin{proof}

    (I) Assume that every $X\in\mathcal{X}$ is finite. By Lemma \ref{lem:finitemin} the set $\mathcal{X}_{\text{min}}$ is finite. Since each subgroup $X\in\mathcal{X}_{\text{min}}$ is finite, it follows from \cite[Lemma~7.9]{Hume} (see also the proof of \cite[Theorem~4.7]{BGHL}) that the normal subgroup $N$ generated by all $X\in\mathcal{X}_{\text{min}}$ is finite. Then $C^*_r(N) = \mathbb{C}[N]$, so it holds trivially that  $\mathcal{X}_{\text{min}} \in \mathcal{D}_{N}$, and Theorems \ref{thm:reducetonsg} and \ref{thm:reduceX} imply $\mathcal{X}\in\mathcal{D}_{\Gamma}$.

    (II) Assume every $X\in\mathcal{X}$ is co-finite. Note that if $\Lambda\subseteq \Gamma$ is a subgroup, then $\Lambda\cap X$ is co-finite in $\Lambda$, for every $\Lambda\cap X\in \Lambda\cap\mathcal{X}$. Therefore, Proposition \ref{prp:induc} allows us to reduce to the special case where $\Gamma$ is finitely generated. Since every co-finite subgroup of a finitely generated group is finitely generated, Lemma \ref{lem:finitemin} implies $\mathcal{X}_{\text{min}}$ is finite. Hence, the normal subgroup $N \coloneq \bigcap_{X\in\mathcal{X}_{\text{min}}} X$ is co-finite. Write $\mathcal{Y} = \{X/N \colon X\in\mathcal{X}_{\text{min}}\}$. Then, $\mathcal{Y}$ is a conjugation-invariant set of subgroups in the finite group $\Gamma/N$ with $q^{*}_{N}(\mathcal{Y}) = \mathcal{X}_{\text{min}}$. Since $\mathcal{Y}\in\mathcal{D}_{\Gamma/N}$ (automatically due to finiteness of $\Gamma/N$), we have $\mathcal{X}_{\text{min}} = q^{*}_{N}(\mathcal{Y})\in\mathcal{D}_{\Gamma}$ by Theorem \ref{thm:quotients}, and hence $\mathcal{X}\in\mathcal{D}_{\Gamma}$ by Theorem \ref{thm:reduceX}.

    (III) Let $\Gamma$ be a countable and finite-by-nilpotent group. Our goal is to show that $\mathcal{X} \in \mathcal{D}_\Gamma$. Every subgroup of $\Gamma$ is finite-by-nilpotent so, by Proposition \ref{prp:induc}, we can assume that $\Gamma$ is finitely generated. Then, by Lemma \ref{lem:fbnlemma}, every subgroup of $\Gamma$ is finitely generated, so Lemma \ref{lem:finitemin} implies that $\mathcal{X}_{\text{min}}$ is finite. By Theorem \ref{thm:reduceX} it therefore suffices to assume that $\mathcal{X}$ is finite.

    We show that $\mathcal{X} \in \mathcal{D}_\Gamma$ by induction on $|\mathcal{X}|$. The base case $|\mathcal{X}| = 1$ is covered by Corollary~\ref{cor:nsg}. Suppose that $|\mathcal{X}| \ge 2$, and assume that whenever a conjugation-invariant collection of subgroups in a finite-by-nilpotent group $\Lambda$ has cardinality less than $|\mathcal{X}|$, then it belongs to $\mathcal{D}_\Lambda$. Let $N \subseteq \Gamma$ be the subgroup generated by all elements of $\mathcal{X}$. If $N$ contains every element of $\mathcal{X}$ as a finite index subgroup, then $\mathcal{X} \in \mathcal{D}_{N}$ by (II) above, and hence $\mathcal{X} \in \mathcal{D}_\Gamma$ by Theorem~\ref{thm:reducetonsg}. Suppose otherwise, and choose $X_0 \in \mathcal{X}$ with infinite index in $N$. Since $\Gamma$ is finite-by-nilpotent, there exists a subnormal subgroup $\Lambda \subseteq \Gamma$ containing $X_0$ as a finite-index subgroup. Note that $N \nsubseteq \Lambda$, by assumption. Let $\Gamma = \Lambda_0 \trianglerighteq \Lambda_1 \trianglerighteq \dots \trianglerighteq \Lambda_n = \Lambda$ be a subnormal series, and let $k \in \{1, \dots, n\}$ be minimal such that $N \nsubseteq \Lambda_k$. Define non-empty subsets $\mathcal{X}_1 \coloneq \{X \in \mathcal{X} \colon X \subseteq \Lambda_k \}$ and $\mathcal{X}_2 = \mathcal{X} \setminus \mathcal{X}_1$. Since $\Lambda_{k}$ is normal in $\Lambda_{k-1}$, both $\mathcal{X}_1$ and $\mathcal{X}_2$ are conjugation-invariant collections of subgroups in the finite-by-nilpotent group $\Lambda_{k-1}$ with cardinality strictly less than $|\mathcal{X}|$. By the inductive hypothesis, both $\mathcal{X}_1$ and $\mathcal{X}_2$ belong to $\mathcal{D}_{\Lambda_{k-1}}$. By Proposition~\ref{prp:finiteunion} we have $\mathcal{X} = \mathcal{X}_1 \cup \mathcal{X}_2 \in \mathcal{D}_{\Lambda_{k-1}}$, and hence $\mathcal{X} \in \mathcal{D}_\Gamma$ by Theorem~\ref{thm:reducetonsg}. This completes the proof of the theorem.
\end{proof}

The following corollary is immediate by Theorems \ref{ReductionToIso}, Theorem \ref{thm:bigthm} and Corollary \ref{cor:arbnsg}, where $\mathcal{X}(x)$ is as defined in Definition~\ref{dfn:mathcalX(x)}.

\begin{cor}\label{cor:bigcor}
    Let $G$ be an amenable and second-countable \'etale groupoid. For each $x\in \Gu$ assume that one of the following holds.
    \begin{enumerate}[label=(\Roman*)]
        \item The isotropy group $G_{x}^{x}$ is finite-by-nilpotent.
        \item The subgroups $X\in \mathcal{X}(x)$ are all finite.
        \item The subgroups $X\in \mathcal{X}(x)$ are all co-finite in $G^{x}_{x}$.
        \item The subgroups $X \in \mathcal{X}(x)$ are all normal in $G^x_x$.
    \end{enumerate}
    Then $J \cap \cg$ is dense in $J$.
\end{cor}

Note that condition (II) holds whenever all elements of $\tiGu_\ess$ are finite. Therefore, within the class of amenable and second-countable \'etale groupoids, the class described in Corollary \ref{cor:bigcor} (II) is larger than the class described in \cite[Theorem~4.7]{BGHL}.

\vspace{0.3cm}
The final result of this section considers \'etale group bundles -- that is, \'etale groupoids whose range and source maps are equal.

\begin{cor}\label{cor:BundleOfGroups}
    Let $G$ be an amenable and second-countable \'etale group bundle. Then $J \cap \cg$ is dense in $J$.
\end{cor}

\begin{proof}
    We show that for every $x \in \Gu$, the subgroups $X \in \mathcal{X}(x)$ are all normal in $G^x_x$. The result then follows by Corollary~\ref{cor:bigcor} (IV).

    Let $x \in \Gu$, $X \in \mathcal{X}(x)$ and choose a net $(x_\alpha) \subset \Gu$ whose set of limit points is equal to $X$. Take $g \in G^x_x$ and let $U \subseteq G$ be an open bisection containing $g$. The $x_\alpha$ eventually belong to $s(U)$, so there exist $g_\alpha \in U$ with $s(g_\alpha) = x_\alpha$. Since $s|_U \colon U \to s(U)$ is a homoeomorphism, the net $(g_\alpha)$ converges to $g$. Then, since $g_\alpha = g_\alpha x_\alpha$ for all $\alpha$, continuity of multiplication on the groupoid ensures that $(g_\alpha)$ has set of limit points equal to $gX$. However, $r(g_\alpha) = s(g_\alpha)$, so we have $g_\alpha = x_\alpha g_\alpha$ for all $\alpha$, and hence continuity of multiplication also ensures that $(g_\alpha)$ has set of limit points equal to $Xg$. Therefore $gX = Xg$. The element $g \in G^x_x$ was arbitrary, so $X$ is a normal subgroup of $G^x_x$, as required.
\end{proof}

\section{Structure of $J\cap\cg$}\label{sec:StructureJcg}

In this section we provide explicit instructions for building elements of $J \cap \cg$, and show that functions constructed in this way have dense linear span in $J \cap \cg$. Therefore, when $J \cap \cg$ is dense in the singular ideal $J$, we are able to describe an explicit family of functions in $\cg$ whose linear span is dense in $J$.

\vspace{0.3cm}
Let us describe the approach. One starts with a unit $x \in \Gu$ and a finitely supported function $b \colon G^x_x \to \mathbb{C}$ that satisfies the linear equations \eqref{eqn:ideal}. Given $h \in G_x$, translate the support of $b$ by left-multiplication, and choose open bisections $V_1, \dots, V_n$ containing these translates. Lastly, choose an open neighbourhood $W \subseteq \bigcap_{i=1}^n s(V_i)$ of $x$ such that the open bisections $U_i \coloneq V_i \cap s^{-1}(W)$ satisfy the conditions of Lemma~\ref{lem:NewWI_condition}. Then, for any $\psi \in C_c(W)$, the function $f^\psi$ (defined in \eqref{f^psi}) belongs to $J \cap \cg$ (see Proposition~\ref{NewCutting}). Overall, the function $f^\psi$ is determined by data in the tuple $(x,h,b,U_1, \dots,U_n,\psi)$. By ranging over all valid tuples, one obtains a family of functions whose linear span is dense in $J \cap \cg$ (see Theorem~\ref{SpanJcg}).

We remark that since $\mathcal{X}(x)$ is invariant under conjugation, an element $b \in \mathbb{C}[G^x_x]$ satisfies \eqref{eqn:ideal} if and only if, for every subgroup $X \in \mathcal{X}(x)$, $b$ belongs to the kernel of the associated quasi-regular representation $\lambda_{G^x_x/X}$.

\subsection{Preliminary lemmas} We will need three preliminary results.

\begin{lem}\label{EqualOnFib}
    Let $G$ be an \'etale groupoid. Fix $f \in \cg$, $x \in \Gu$, and assume that $f(g) = 0$ for all $g \in G_x$. Then $f$ is continuous at each $g \in G_x$.
\end{lem}

\begin{proof}
    Let $g \in G_x$ and let $(g_\beta)$ be a net in $G$ converging to $g$. In order to prove that $\lim_\beta f(g_\beta) = 0$, it suffices to find a subnet $(g_\gamma)$ satisfying $\lim_\gamma f(g_\gamma) = 0$ (since $(g_\beta)$ is arbitrary). There exists a subnet $(g_\gamma)$ for which the net $\iota(g_\gamma)$ converges in $\tiG$ (i.e. in the Fell topology). Let $\bm{g} \in \tiG$ denote the limit. By \eqref{Cond1} we have $\mfi(f)(\bm{g}) = \sum_{g \in \bm{g}} f(g) = 0$ since $\bm{g} \subseteq G_x$. Moreover, $f(g_\gamma) = \mfi(f)\big(\iota(g_\gamma)\big) \to \mfi(f)(\bm{g})$ by continuity of $\mfi(f)$. This completes the proof.
\end{proof}

The following results isolate observations from \cite[Proposition~5.19]{Hume}. We follow the same proofs (see also \cite[Theorem~4.7]{BGHL}). Given $x \in \Gu$, recall that a subgroup $X \subseteq G^x_x$ belongs to $\mathcal{X}(x)$ if and only if there exists a net $(x_\alpha) \subset C$ of (Hausdorff) units whose set of limit points is precisely $X$, and such that every subnet of $(x_{\alpha})$ has its limit points contained in $X$ (see Definition~\ref{dfn:mathcalX(x)}).

\begin{lem}\label{lem:NewWI_condition}
    Let $G$ be an \'etale groupoid that is covered by countably many open bisections.  Fix $x \in \Gu$, and let $g_1, \dots, g_n \in G_{x}$ be distinct points.
    
    There exist open bisections $U_1, \dots, U_n \subseteq G$ satisfying the following conditions:
    \begin{enumerate}[label=(\roman*)]
        \item $g_i \in U_i$ for all $i \in \{1, \dots, n\}$;
        \item $s(U_i) = s(U_j)$ for all $i, j \in \{1, \dots, n\}$;
        \item $W_I = \emptyset$ whenever $I \subseteq \{1, \dotsc, n\}$ is \underline{not} of the form
        \begin{equation}\label{eq:WI_NiceForm}
        \{g_i \colon i \in I\} = kX\cap \{g_1, \dots, g_n\}
        \end{equation}
        for some $k \in G_{x}$ and $X \in \mathcal{X}(x)$.
    \end{enumerate}
    Here $W_I \coloneq s\big( \bigcap_{i \in I} U_i \setminus \bigcup_{j \notin I} U_j \big) \cap C$.    
\end{lem}

Note that for any small open set $U \subseteq \Gu$ containing $x$, the cutdowns $U_i' \coloneq U_i \cap s^{-1}(U)$ also satisfy (i) -- (iii).

\begin{proof}
    It is clear that there exist open bisections $V_1, \dots, V_n \subseteq G$ satisfying the first two conditions. Write $V \coloneq s(V_1)$ and $V_I \coloneq s\big( \bigcap_{i \in I} U_i \setminus \bigcup_{j \notin I} U_j \big) \cap C$ for subsets $I \subseteq \{1, \dots, n\}$. Choose an open subset $W \subseteq \Gu$ such that $x \in W \subseteq V$ and $W \cap \overline{V_I}^{\Gu} = \emptyset$ for all $I \subseteq \{1, \dotsc, n\}$ with $x \notin \overline{V_I}^{\Gu}$. Define $U_i \coloneq V_i \cap s^{-1}(W)$ for $i \in \{1, \dots, n\}$. It is clear that the open bisections $U_1, \dots, U_n \subseteq G$ satisfy conditions (i) and (ii). We claim that they also satisfy condition (iii).

    Let $I \subseteq \{1, \dots, n\}$ be such that $W_I \neq \emptyset$. We claim that $I$ is of the form in \eqref{eq:WI_NiceForm}. It is clear that $W_I = W \cap V_I$ and hence $x \in \overline{W_I}^{\Gu}$ by the construction of $W$. Let $(x_\alpha)$ be a net in $W_I$ converging to $x$. By passing to a subnet if necessary we can assume that the net $\iota(x_\alpha)$ converges in the Hausdorff cover groupoid $\tilde{G}$. Let $X \in \tiGu_\ess$ denote the limit. Note that $X \in \mathcal{X}(x)$ by definition. Also, by definition of the Fell topology, the set of limit points of the net $(x_\alpha)$ is precisely $X$. Now, let $g_\alpha \in \bigcap_{i \in I} U_i \setminus \bigcup_{j \notin I} U_j$ be the unique elements satisfying $s(g_\alpha) = x_\alpha$. Continuity of multiplication on the groupoid implies that the net $(g_\alpha)$ has set of limit points $g_{i_0} X$, where $i_0$ is any element of $I$. The unit $x$ belongs to $X$, and hence $\{g_i \colon i \in I\} \subseteq g_{i_0} X$. On the other hand, if $j \notin I$ then $g_\alpha \notin U_j$ and hence the net $(g_\alpha)$ does not converge to $g_j$. Therefore $g_j \notin g_{i_0} X$, and it follows that $\{g_i \colon i \in I\} = g_{i_0}X\cap \{g_1, \dots, g_n\}$ for any $i_0 \in I$. This completes the proof.
\end{proof}

The next proposition is a slight generalisation of the construction of singular elements found in \cite[Proposition~5.19]{Hume}. The proof is exactly the same, but we record it for completeness.

\begin{prp}\label{NewCutting}
    Let $G$ be an \'etale groupoid that is covered by countably many open bisections. Fix $x \in \Gu$, and suppose $b \in \mathbb{C}[G^x_x]$ belongs to $\ker(\lambda_{G^x_x/X})$ for all $X \in \mathcal{X}(x)$. Let $\{g_1, \dots, g_n\} \subseteq G^x_x$ denote the support of $b$, let $h \in G_x$, and let $U_1, \dots, U_n$ be open bisections containing $hg_1, \dots, hg_n$ and satisfying the conditions of Lemma~\ref{lem:NewWI_condition}. Then
    \begin{equation}
    \label{f^psi}
    f^\psi \coloneq \sum_{i = 1}^n b(g_i) (\psi \circ s \vert_{U_i})
    \end{equation}
    belongs to $J \cap \cg$ for any $\psi \in C_c\big(s(U_i)\big)$ (note that this does not depend on $i$).
\end{prp}

\begin{proof}
    For $I \subseteq \{1, \dots, n\}$ set $\check{S}_I = (\bigcap_{i \in I} U_i) \setminus (\bigcup_{j \notin I} U_j)$ and $W_I = s(\check{S}_I) \cap C$, as in Lemma~\ref{lem:NewWI_condition}.
    Take arbitrary $\psi \in C_c\big(s(U_i)\big)$, and let $f^\psi$ be as in \eqref{f^psi}. If $W_I \neq \emptyset$, Lemma~\ref{lem:NewWI_condition}(iii) ensures that $\{hg_i \colon i \in I\} = kX\cap \{hg_1, \dots, hg_n\}$ for some $k \in G_x$ and $X \in \mathcal{X}(x)$. Clearly $h^{-1}k \in G^x_x$. Therefore, for $g \in \check{S}_I$ we have
    \[
    f^\psi(g) = \sum_{i \in I} b(g_i) \psi(s(g)) = \sum_{t \in h^{-1}kX} b(t) \psi(s(g)) = 0
    \]
    since $b \in \ker(\lambda_{G^x_x / X})$. It follows that any $g \in \osupp(f^\psi)$ belongs to $\check{S}_I$ for some $I \subseteq \{1, \dots, n\}$ with $W_I = \emptyset$, and hence $s\big(\osupp(f^\psi)\big) \subseteq \Gu \setminus C$. This implies that $f^\psi \in J$ by \cite[Proposition~7.18]{KM21}, as desired.
\end{proof}

\PRLsep

\subsection{Building blocks for $J \cap \cg$}\label{subsec:BuildingBlocksJcg}
Let $G$ be an \'etale groupoid that is covered by countably many open bisections. Each building block is assembled from the data in a tuple $\Xi \coloneq (x, h, b, U_1, \dots, U_n, \psi)$ where
\begin{itemize}
    \item $x \in \Gu$;
    \item $h \in G_x$;
    \item $b \in \mathbb{C}[G^x_x]$ belongs to $\ker(\lambda_{G^x_x/X})$ for all $X \in \mathcal{X}(x)$, and has support $\{g_1, \dots, g_n\}$;
    \item $U_1, \dots, U_n \subseteq G$ are open bisections satisfying the conditions of Lemma~\ref{lem:NewWI_condition}. That is, they have equal source, contain the elements $hg_1, \dots, hg_n \in G_x$, and are such that $W_I = \emptyset$ whenever $I \subseteq \{1, \dotsc, n\}$ is \underline{not} of the form
        \[
        \{hg_i \colon i \in I\} = kX\cap \{hg_1, \dots, hg_n\} \quad \quad \text{for some } k \in G_{x}, X \in \mathcal{X}(x).
        \]
    Here $W_I \coloneq s\big( \bigcap_{i \in I} U_i \setminus \bigcup_{j \notin I} U_j \big) \cap C$. 
    \item $\psi \in C_c\big(s(U_i)\big)$ (note that this does not depend on $i$).
\end{itemize}
Given such a tuple $\Xi$ we define
\begin{equation}\label{eq:building_blocks}
f^\Xi \coloneq \sum_{i=1}^n b(g_i) (\psi \circ s \vert_{U_i}).
\end{equation}
By Proposition~\ref{NewCutting} we have $f^\Xi \in J \cap \cg$. Let $\mathcal{E} \subseteq J \cap \cg$ denote the linear span of elements of the form $f^\Xi$ from \eqref{eq:building_blocks}, where we let $\Xi$ range over all valid tuples. We will see that $\mathcal{E}$ is dense in $J \cap \cg$.

\begin{lem}\label{lem:RestrictJcg}
    Let $G$ be an \'etale groupoid that is covered by countably many open bisections.
    Fix $f \in J \cap \cg$, $x \in \Gu$, and let $g_1, \dotsc, g_n \in G_{x}$ be all elements of $G_{x}$ on which $f$ is non-zero. There exist open bisections $U_1, \dots, U_n$ with equal source, containing $g_1, \dots, g_n$ and such that
    \[
    f^\psi \coloneq \sum_{i = 1}^n f(g_i) (\psi \circ s \vert_{U_i})
    \]
    belongs to $\mathcal{E}$ for any $\psi \in C_c\big(s(U_i)\big)$ (note that this does not depend on $i$).
\end{lem}

\begin{proof}
    Let $U_1, \dots, U_n$ be open bisections satisfying the conditions of Lemma~\ref{lem:NewWI_condition} and such that $U_{i_1}$ and $U_{i_2}$ are disjoint whenever $r(g_{i_1}) \neq r(g_{i_2})$. Partition the set $\{g_1, \dots, g_n\}$ according to the images under the range map $r$ and write $\{g_1, \dots, g_n\} = \bigsqcup_{j=1}^m \{g^j_1, \dots, g^j_{k_j}\}$, $\{U_1, \dots, U_n\} = \bigsqcup_{j=1}^m \{U^j_1, \dots, U^j_{k_j}\}$, where $g_i^j \in U_i^j$. Fix $\psi \in C_c\big(s(U_i)\big)$. For each $j \in \{1, \dots, m\}$ define
    \[
    f^j \coloneq \sum_{i = 1}^{k_j} f\big( g^j_i \big) \big(\psi \circ s \vert_{U^j_i}\big)
    \]
    so that $f^{\psi} = \sum_{j=1}^m f^j$. For each $j$, we show that $f^j = f^{\Xi_j}$ for some valid tuple $\Xi_j$, where $f^{\Xi_j}$ is as in \eqref{eq:building_blocks}. The lemma then follows.
    
    Let $V_j$ be an open bisection containing ${(g^j_1)}^{-1}$, and let $\varphi \in C_c(V_j)$ be such that $\varphi\big(({g^j_1}\big)^{-1}) = 1$. Define $b_j \coloneq \eta_x(\varphi * f)$, where $ \eta_x \colon \cg \to \mathbb{C}[G^x_x]$ is the restriction map. Since $\varphi * f \in J \cap \cg$, it follows from \cite[Proposition~5.19]{Hume} that $b_j$ belongs to $\ker(\lambda_{G^x_x/X})$ for all $X \in \mathcal{X}(x)$. Then, since $f(g_i^j) = b_j\big( {(g^j_1)}^{-1} g^j_i \big)$ for all $i = 1, \dots, k_j$, we have $f^j = f^{\Xi_j}$ where $\Xi_j \coloneq \big(x, g_1^j, b_j, U^j_1, \dots, U^j_{k_j}, \psi \big)$. The initial assumptions on the bisections $U_1, \dots, U_n$ ensure that $\Xi_j$ is a valid tuple.
\end{proof}

For $f \in \cg$, define
\begin{equation}\label{eq:Inorm}
\norm{f}_I \coloneq \max\left\{ \sup_{x \in \Gu} \sum_{g \in G_x} \abs{f(g)} , \sup_{x \in \Gu} \sum_{g \in G^x} \abs{f(g)} \right\},
\end{equation}
following \cite{Hahn}.

\begin{thm}\label{SpanJcg}
    Let $G$ be an \'etale groupoid that is covered by countably many open bisections. Then $\mathcal{E}$ is dense in $J \cap \cg$ with respect to the norm $\norm{\cdot}_I$.
\end{thm}

The norm $\norm{\cdot}_I$ dominates the full norm $\norm{\cdot}_{C^*(G)}$ and hence any $C^{*}$-norm. Therefore, $\mathcal{E}$ is dense in $J \cap \cg$ with respect to any $C^{*}$-norm.

\begin{proof}
    Let $f \in J \cap \cg$. Take a compact set $K \subseteq G$ with $\osupp(f)$ contained in its interior. Whenever $\dot{f} \in \cg$ satisfies $\osupp(\dot{f}) \subseteq K$ we have $\norm{\dot{f}}_{I} \le C_K \norm{\dot{f}}_\infty
    $ where $\norm{\cdot}_\infty$ denotes the supremum norm, and $C_K$ is a constant depending only on $K$. Therefore, it suffices to find, for each $\varepsilon>0$, a function $f' \in \mathcal{E}$ satisfying $\osupp(f') \subseteq K$ and $\norm{f-f'}_\infty \le \varepsilon$.
    
    Fix $\varepsilon>0$. Take $x \in s(K)$, and let $g_1, \dotsc, g_n \in G_x$ be all elements of $G_x$ on which $f$ is non-zero. Let $U_1, \dots, U_n \subseteq K$ be open bisections as in Lemma~\ref{lem:RestrictJcg}, and define $U_x \coloneq s(U_i)$ (note that this does not depend on $i$). In particular $x \in U_x$. Then
    \[
    f^{\psi} \coloneq \sum_{i=1}^n f(g_i) (\psi \circ s \vert_{U_i})
    \]
    belongs to $\mathcal{E}$ for any $\psi \in C_c(U_x)$. Fix some $\psi_x \in C_c(U_x)$ with $\psi_x(x) = 1$. Clearly $\osupp(f^{\psi_x}) \subseteq K$. We have $f|_{G_{x}} = f^{\psi_x}|_{G_{x}}$ and $\osupp(f-f^{\psi_x}) \subseteq K$, so by Lemma \ref{EqualOnFib} there exists an open neighbourhood $W_x \subseteq U_x$ of $x$ such that $\abs{f(g) - f^{\psi_x}(g)} \le \varepsilon$ whenever $s(g) \in W_x$. The sets $\{W_x\}_{x \in s(K)}$ form an open cover for the compact set $s(K)$, so select a finite subcover $W_{x_1}, \dotsc, W_{x_m}$. Let $\varphi_1, \dots, \varphi_m \in C_c\big(\Gu\big)$ be a partition of unity for $s(K)$ subordinate to this cover, and define
    \[
    f' \coloneq \sum_{j=1}^m (\varphi_j \circ s) f^{\psi_{x_j}} = \sum_{j=1}^m f^{\varphi_j.\psi_{x_j}}.
    \]
    Then $f' \in \mathcal{E}$ satisfies $\osupp(f') \subseteq K$ and $\norm{f-f'}_\infty \le \varepsilon$, as required.
\end{proof}

\begin{rmk}\label{rmk:UncountablyBisections}
    The construction in this subsection generalises to groupoids $G$ that are not covered by countably many open bisections. For an open subgroupoid $H \subseteq G$ that \emph{is} covered by countably many open bisections, define $\mathcal{E}_H$ to be the linear span of functions of the form $f^\Xi$ in \eqref{eq:building_blocks} for tuples $\Xi$ in $H$. Then, define $\mathcal{E} \subseteq J \cap \cg$ to be the subspace generated by all the $\mathcal{E}_H$. 
    With this modification, the conclusion of Theorem~\ref{SpanJcg} holds for any \'etale groupoid.
\end{rmk}

\section{Ample Groupoids}
\label{sec:Ample}

In this section we describe an explicit spanning set for the algebraic singular ideal $J_\mathbb{C}$ of an ample groupoid. The construction is similar to that in Subsection~\ref{subsec:BuildingBlocksJcg}. In Subsection~\ref{subsec:AlgSingDenseJcg} we prove that the algebraic singular ideal $J_\mathbb{C}$ is dense in $J \cap \cg$ for any ample groupoid.

\vspace{0.3cm}
Recall that an \emph{ample groupoid} is an \'etale groupoid with a basis of compact open bisections. If $G$ is an ample groupoid, the \emph{Steinberg algebra} $\mathbb{C}G$ over the complex numbers is defined in \cite{S10} as
\[
\mathbb{C}G \coloneq \text{span}\left\{ \mathbbm{1}_U \colon U \subseteq G \text{ a compact open bisection}\right\}.
\]
The \emph{algebraic singular ideal} is the intersection $J_{\mathbb{C}} \coloneq J \cap \mathbb{C}G$. Since $\mathbb{C}G \subseteq \cg$, we have $J_\mathbb{C} \subseteq J \cap \cg$.

\PRLsep

\subsection{Building blocks for $J_\mathbb{C}$}\label{subsec:BuildingBlocksAlgSing} Let $G$ be an ample groupoid that is covered by countably many open bisections. Each building block is assembled from the data in a tuple $\Theta \coloneq (x, h, b, U_1, \dots, U_n)$ where
\begin{itemize}
    \item $x \in \Gu$;
    \item $h \in G_x$;
    \item $b \in \mathbb{C}[G^x_x]$ belongs to $\ker(\lambda_{G^x_x/X})$ for all $X \in \mathcal{X}(x)$, and has support $\{g_1, \dots, g_n\}$;
    \item $U_1, \dots, U_n \subseteq G$ are compact open bisections satisfying the conditions of Lemma~\ref{lem:NewWI_condition}. That is, they have equal source, contain the elements $hg_1, \dots, hg_n \in G_x$, and are such that $W_I = \emptyset$ whenever $I \subseteq \{1, \dotsc, n\}$ is \underline{not} of the form
        \[
        \{hg_i \colon i \in I\} = kX\cap \{hg_1, \dots, hg_n\} \quad \quad \text{for some } k \in G_{x}, X \in \mathcal{X}(x).
        \]
    Here $W_I \coloneq s\big( \bigcap_{i \in I} U_i \setminus \bigcup_{j \notin I} U_j \big) \cap C$. 
\end{itemize}
Given such a tuple $\Theta$ we define
\begin{equation}\label{eq:building_blocks_AlgSing}
f^\Theta \coloneq \sum_{i=1}^n b(g_i) \mathbbm{1}_{U_i}.
\end{equation}
Then $f^\Theta \in J_{\mathbb{C}}$ by Proposition~\ref{NewCutting}.

\begin{thm}\label{SpanAlgSing}
    Let $G$ be an ample groupoid that is covered by countably many open bisections. The algebraic singular ideal $J_\mathbb{C}$ is linearly spanned by elements of the form $f^\Theta$ from \eqref{eq:building_blocks_AlgSing}.
\end{thm}

\begin{proof}
    Let $f \in J_\mathbb{C}$. Take a compact set $K \subseteq G$ with $\osupp(f) \subseteq K$. Fix $x \in s(K)$, and let $g_1, \dots, g_n$ be all elements of $G_x$ on which $f$ is non-zero. An argument identical to that in Lemma~\ref{lem:RestrictJcg} yields compact open bisections $U_1, \dots, U_n$ containing $g_1, \dots, g_n$, with equal source, and such that
    \[
    f^x \coloneq \sum_{i=1}^n f(g_i) \mathbbm{1}_{U_i}
    \]
    is a finite sum of functions of the form $f^\Theta$ from \eqref{eq:building_blocks_AlgSing}. Let $U_x \coloneq s(U_i)$ (note that this does not depend on $i$). In particular $x \in U_x$. We have $f|_{G_{x}} = f^x|_{G_{x}}$, so, by Lemma \ref{EqualOnFib}, $f-f^x$ is continuous at each $g \in G_x$. However, both $f$ and $f^x$ belong to $\mathbb{C}G$, and hence only take finitely many values. Therefore, by shrinking the $U_i$ if necessary, we can assume that $f|_{U_x} = f^x|_{U_x}$ (one can always shrink the compact open bisections in a tuple $\Theta$). The sets $\{U_x\}_{x \in s(K)}$ form an open cover for the compact set $s(K)$, so select a finite subcover $U_{x_1}, \dotsc, U_{x_m}$. Removing intersections if necessary, we may assume that the $U_{x_j}$ are pairwise disjoint (this is possible since the $U_{x_j}$ are contained in the Hausdorff space $\Gu$, and are thus clopen in $\Gu$). Define
    \[
    f' \coloneq \sum_{j=1}^m f^{x_j} .
    \]
    It is clear that $f'$ is a finite sum of functions of the form $f^\Theta$ from \eqref{eq:building_blocks_AlgSing}. This completes the proof, since $f' = f$ by construction.
\end{proof}

\begin{rmk}
    The construction in this subsection generalises to ample groupoids that are not covered by countably many open bisections. After a modification analogous to that in Remark~\ref{rmk:UncountablyBisections}, one obtains, for any ample groupoid, and explicit spanning set for $J_\mathbb{C}$.
\end{rmk}

\PRLsep

\subsection{Density of $J_\mathbb{C}$ in $J \cap \cg$}\label{subsec:AlgSingDenseJcg} We prove that the algebraic singular ideal $J_\mathbb{C}$ is dense in $J \cap \cg$ for any ample groupoid -- no countability assumptions are needed. In Corollary~\ref{SelfSim} this observation is applied to the class of groupoids arising from contracting self-similar groups.

\begin{thm}\label{AlgSingId}
    Let $G$ be any ample groupoid. Then $J_{\mathbb{C}}$ is dense in $J \cap \cg$ with respect to the norm $\norm{\cdot}_I$.
\end{thm}

As stated earlier (see \eqref{eq:Inorm}), the norm $\norm{\cdot}_I$ dominates the full norm $\norm{\cdot}_{C^*(G)}$ and hence any $C^{*}$-norm. Therefore, $J_{\mathbb{C}}$ is dense in $J \cap \cg$ with respect to any $C^{*}$-norm.

\begin{proof}
    Given $f \in J \cap \cg$, write $f = \sum_{j=1}^m f_j$ for some $f_j \in C_c(V_j)$ and open bisections $V_j$, and let $H \subseteq G$ be the subgroupoid generated by the open bisections $V_1, \dots, V_m$. Then $H$ is an open subgroupoid that is covered by countably many open bisections. By working in $H$ instead, we can assume that the groupoid $G$ is covered by countably many open bisections. Then, by Theorem~\ref{SpanJcg}, it suffices to assume that $f = f^\Xi$ for some tuple $\Xi = (x, h, b, U_1, \dots, U_n, \psi)$ as in Subsection~\ref{subsec:BuildingBlocksJcg}, and where
    \[
    f^\Xi = \sum_{i=1}^n b(g_i) (\psi \circ s|_{U_i}),
    \]
    as in \eqref{eq:building_blocks}. Take a compact set $K \subseteq G$ containing each $U_i$. As in the proof of Theorem~\ref{SpanJcg} it suffices to find, for each $\varepsilon>0$, a function $f' \in J_{\mathbb{C}}$ satisfying $\osupp(f') \subseteq K$ and $\norm{f^\Xi-f'}_\infty \le \varepsilon$.
    
    Fix $\varepsilon>0$. Define $U \coloneq s(U_i)$ (note that this does not depend on $i$). Since $U$ is totally disconnected, there exists a locally constant function $\varphi \in C_c(U)$ satisfying $\norm{\psi - \varphi}_\infty \le \frac{\varepsilon}{n \max_i(\abs{b(g_i)})}$. Consider the function
    \[
    f' \coloneq \sum_{i=1}^n b(g_i) (\varphi \circ s|_{U_i}).
    \]
    It is clear that $f'$ belongs to the Steinberg algebra $\mathbb{C}G$, and satisfies the conditions $\osupp(f') \subseteq K$ and $\norm{f^\Xi-f'}_\infty \le \varepsilon$. Moreover, by Proposition~\ref{NewCutting}, $f'$ belongs to the singular ideal $J$, and hence also to $J_\mathbb{C}$. This completes the proof.
\end{proof}

We apply the previous theorem to the class of groupoids arising from contracting self-similar groups.

\begin{cor}\label{SelfSim}
    Let $G$ be the groupoid arising from a contracting self-similar group action. Then $J_\mathbb{C}$ is dense in $J$.
\end{cor}

\begin{proof}
    By Theorem \ref{AlgSingId} it suffices to show that $J \cap \cg$ is dense in $J$. By \cite[Corollary~2.19]{MilSte25} (see also \cite{Nek09} and \cite{GNSV}) the groupoid $G$ is second-countable and amenable. Moreover, \cite[Corollary~7.13]{BGHL} implies that every element of $\tiGu$ is finite. The result follows by Corollary~\ref{cor:bigcor} (II).
\end{proof}

\begin{rmk}
    We would like to thank B. Kwaśniewski for pointing out to us that the previous result generalises to groupoids $G$ arising from contracting self-similar \emph{groupoid} actions on finite graphs without sources. Such a groupoid is amenable by \cite[Corollary~2.19]{MilSte25}, and it follows from either \cite[Corollary~6.23]{KM25} or \cite[Proposition~9.1]{Aakre} that elements of $\tiGu$ are finite.
\end{rmk}


\begin{thebibliography}{99}

\bibitem{Aakre} J. \textsc{Aakre}, \emph{Simplicity of algebras and $C^*$-algebras of self-similar groupoids}, J. Math. Anal. Appl. \emph{559} (2026), no. 1, 130446.

\bibitem{BKM} K. \textsc{Bardadyn}, B. \textsc{Kwaśniewski} and A. \textsc{McKee}, \emph{Banach algebras associated to twisted \'etale groupoids: simplicity and pure infiniteness}, preprint, arXiv:2406.05717.

\bibitem{BHV07} B. \textsc{Bekka}, P. \textsc{de la Harpe} and A. \textsc{Valette}, \emph{Kazhdan’s Property (T)}, 
    New Math. Monogr., \emph{11},  Cambridge Univ. Press, 2008.

\bibitem{BK20} B. \textsc{Bekka} and M. \textsc{Kalantar}, \emph{Quasi-regular representations of discrete groups and associated $C^*$-algebras}, Trans. Amer. Math. Soc. \emph{373} (2020), 2105--2133.


\bibitem{BGHL} K. A. \textsc{Brix}, J. \textsc{Gonzales}, J. B. \textsc{Hume} and X. \textsc{Li}, \emph{On Hausdorff covers for non-Hausdorff groupoids}, preprint, arXiv:2503.23203.

\bibitem{BO08} N. P. \textsc{Brown} and N. \textsc{Ozawa}, \emph{$C^*$-algebras and finite-dimensional approximations}, Grad. Stud. Math., \emph{88}, Amer. Math. Soc., Providence, RI, 2008.

\bibitem{BM25} A. \textsc{Buss} and D. \textsc{Mart{\'i}nez}, \emph{Essential groupoid amenability and nuclearity of groupoid $C^*$-algebras}, J. Funct. Anal. \emph{290} (2026), no. 11, 111427.


\bibitem{CN22} J. \textsc{Christensen} and S. \textsc{Neshveyev}, \emph{(Non)exotic completions of the group algebras of isotropy groups},
Int. Math. Res. Not. IMRN \emph{19} (2022), 15155--15186.

\bibitem{CN24} J. \textsc{Christensen} and S. \textsc{Neshveyev}, \emph{Isotropy fibers of ideals in groupoid $C^*$-algebras}, Adv. Math. \emph{447} (2024), Paper No. 109696.

\bibitem{CEPSS} L. O. \textsc{Clark}, R. \textsc{Exel}, E. \textsc{Pardo}, A. \textsc{Sims} and C. \textsc{Starling}, \emph{Simplicity of algebras associated to non-Hausdorff groupoids}, Trans. Amer. Math. Soc. \emph{372} (2019), no. 5, 3669--3712.

\bibitem{Con82} A. \textsc{Connes}, \emph{A survey of foliations and operator algebras}, operator algebras and applications, part 1 (Kingston, Ont., 1980),  521--628, Proc. Sympos. Pure Math., \emph{38}, American Mathematical Society, Providence, RI, 1982.

\bibitem{Davidson} K. R. \textsc{Davidson}, \emph{$C^*$-algebras by example}, Fields Inst. Monogr.,  Amer. Math. Soc., Providence, RI, 1996.

\bibitem{EP22} R. \textsc{Exel} and D. R. \textsc{Pitts}, \emph{Characterizing groupoid $C^*$-algebras of non-Hausdorff {\'e}tale groupoids}, Lecture Notes in Math., \emph{2306}, Springer, Cham, 2022. 

\bibitem{Fell} J. M. G. \textsc{Fell}, \emph{A Hausdorff topology for the closed subsets of a locally compact non-Hausdorff space}, Proc. Amer. Math. Soc. \emph{13} (1962), 472--476.

\bibitem{GNSV} E. \textsc{Gardella}, V. \textsc{Nekrashevych}, B. \textsc{Steinberg} and A. \textsc{Vdovina}, \emph{Simplicity of $C^*$-algebras of contracting self-similar groups}, Commun. Math. Phys. \emph{406} (2025), 251.

\bibitem{Hahn} P. \textsc{Hahn}, \emph{The regular representations of measure groupoids}, Trans. Amer. Math. Soc. \emph{242} (1978), 35--72.

\bibitem{H54} P. Hall, \emph{Finiteness conditions for soluble groups}, Proc. Lond. Math. Soc. (3) \emph{4} (1954), 419--436.

\bibitem{HLS} N. \textsc{Higson}, V. \textsc{Lafforgue}, and G. \textsc{Skandalis}, \emph{Counterexamples to the Baum-Connes conjecture},
Geom. Funct. Anal. \emph{12} (2002), 330--354.

\bibitem{Hume} J. B. \textsc{Hume}, \emph{Characterization of zero singular ideal in {\'e}tale groupoid $C^*$-algebras via compressible maps}, preprint, arXiv:2509.07262.

\bibitem{IW08} M. \textsc{Ionescu} and D. P. \textsc{Williams}, \emph{The generalized Effros-Hahn conjecture for groupoids}, Indiana Univ. Math. J. \emph{58} (2009), 2489--2508.

\bibitem{KKLRU} M. \textsc{Kennedy}, S. J. \textsc{Kim}, X. \textsc{Li}, S. \textsc{Raum} and D. \textsc{Ursu}, \emph{The ideal intersection property for essential groupoid $C^*$-algebras}, preprint, arXiv:2107.03980.

\bibitem{KM21} B. K. \textsc{Kwa{\'s}niewski} and R. \textsc{Meyer}, \emph{Essential crossed products for inverse semigroup actions: simplicity and pure infiniteness}, Doc. Math. \emph{26} (2021), 271--335.

\bibitem{KM25} B. K. \textsc{Kwa{\'s}niewski} and A. \textsc{Mundey}, \emph{Twisted operator algebras of self-similar groupoid actions on arbitrary graphs}, preprint, arxiv:2511.07906.

\bibitem{LS87} J. C. Lennox and S. E. Stonehewer, \emph{Subnormal subgroups of groups}, Oxford Univ. Press, Oxford, 1987.

\bibitem{MS25} D. \textsc{Mart\'inez} and N. \textsc{Szak\'acs}, \emph{Algebraic singular functions are not always dense in the ideal of $C^*$-singular functions}, preprint, arXiv:2510.01947.

\bibitem{MilSte25} A. \textsc{Miller} and B. \textsc{Steinberg}, \emph{Homology and K-theory for self-similar actions of groups and groupoids}, preprint, arxiv:2409.02359v3.

\bibitem{MM03} I. \textsc{Moerdijk} and J. \textsc{Mr\v cun}, \emph{Introduction to Foliations and Lie Groupoids}, Cambridge Stud. Adv. Math., \emph{91}, Cambridge Univ. Press, Cambridge, 2003.

\bibitem{Nek05} V. \textsc{Nekrashevych}, \emph{Self-similar groups}, Math. Surveys Monogr. \emph{117}, Amer. Math. Soc., Providence, RI, 2005.

\bibitem{Nek09} V. \textsc{Nekrashevych}, \emph{$C^*$-algebras and self-similar groups}, J. Reine Agnew. Math. \emph{630} (2009), 59--123.

\bibitem{NS23} S. \textsc{Neshveyev} and G. \textsc{Schwartz}, \emph{Non-Hausdorff {\'e}tale groupoids and $C^*$-algebras of left cancellative monoids}, M{\"u}nster J. Math. \emph{16} (2023), no. 1, 147--175.

\bibitem{Ren} J. \textsc{Renault}, \emph{A groupoid approach to $C^*$-algebras}, Lecture Notes in Math., \emph{793}, Springer, Berlin, 1980.

\bibitem{RL00} M. Rørdam, F. Larsen and N. Laustsen, \emph{An Introduction to K-Theory for $C^{*}$-Algebras.} Cambridge Univ. Press, 2000.

\bibitem{S10} B. \textsc{Steinberg}, \emph{A groupoid approach to discrete inverse semigroup algebras}, Adv. Math. \emph{223} (2010), no. 2, 689--727.

\bibitem{SS21} B. \textsc{Steinberg} and N. \textsc{Szak{\'a}cs}, \emph{Simplicity of inverse semigroup and {\'e}tale groupoid algebras}, Adv. Math. \emph{380} (2021), Paper No. 107611, 55 pp.

\bibitem{SS23} B. \textsc{Steinberg} and N. \textsc{Szak{\'a}cs}, \emph{On the simplicity of Nekrashevych algebras of contracting self-similar groups}, Math. Ann. \emph{386} (2023), no. 3-4, 1391--1428.

\bibitem{Tim} T. \textsc{Timmermann}, \emph{The Fell compactification and non-Hausdorff groupoids}, Math. Z. \emph{269} (2011), no. 3-4, 1105--1111.

\bibitem{Wil15} R. \textsc{Willett}, \emph{A non-amenable groupoid whose maximal and reduced $C^*$-algebras are the same}, M{\"u}nster J. Math. \emph{8}
(2015), no. 1, 241--252.

\end{thebibliography}
\end{document}